\documentclass[12pt,amstex,amssymb]{article}
\usepackage[margin=0.79in]{geometry}
\usepackage{amsmath,amssymb,latexsym, amsfonts, amscd, amsthm}
\usepackage{graphicx,epsfig}
\usepackage{xcolor}
\usepackage{
	stackrel, tikz-cd, graphicx, etoolbox 
}
\usepackage{hyperref}
\usepackage{enumerate}
\usepackage{bm}

\theoremstyle{plain}
\newtheorem{theorem}{Theorem}[subsection]

\newtheorem{proposition}[theorem]{Proposition}
\newtheorem{corollary}[theorem]{Corollary}

\newtheorem{lemma}[theorem]{Lemma}

\newcommand{\PP}{\mathbb{P}}
\theoremstyle{definition}
\newtheorem{definition}[theorem]{Definition}
\newtheorem{remark}[theorem]{Remark}

\def\lra{\longrightarrow}
\def\g{\gamma}

\def\cB{{\cal B}}
\def\cC{{\cal C}}

\def\cE{{\cal E}}

\def\cR{{\cal R}}

\def\cS{{\cal S}}

\def\fH{{\frak H}}

\def\cM{{\cal M}}
\def\cO{\mathcal {O}}

\def\Z{{\mathbb Z}}
\def\R{{\mathbb R}}

\def\C{{\mathbb C}}

\def\N{{\mathbb N}}
\def\Q{{\mathbb Q}}

\def\Hom{\text{Hom}}
\def\Aut{\text{Aut}}

\def\ker{\text{ker}}
\def\Ker{\text{Ker}}

\def\Vect0{\text{Vect}_0}

\def\bx {{\bf x}}
\def\by {{\bf y}}

\def\ulambda {{\underline{\lambda}}}
\def\sign {{\rm sign}}

\def\BP{{\mathbf{P}}}
\def\l{\lambda} 
\def\d{\delta}
\def\g{\gamma}

  \newcommand{\be}{\begin{eqnarray*}}
  \newcommand{\ee}{\end{eqnarray*}}
  \newcommand{\bea}{\begin{eqnarray}}
  \newcommand{\eea}{\end{eqnarray}}
  
  \newcommand{\Res}{\operatorname{Res}}
  
  \newcommand{\GL}{\operatorname{GL}}
  \newcommand{\SL}{\operatorname{SL}}
   
 \newcommand{\ba}{\mathbf{a}}
  \newcommand{\bb}{\mathbf{b}}
\newcommand{\tfH}{\tilde{\mathfrak{H}}}
\newcommand{\bc}{{\bf c}}

\newcommand{\Dist}{\operatorname{Dist}}
\newcommand{\Spec}{\operatorname{Spec}}
\definecolor{red}{RGB}{255,0,0}

\begin{document}

\title{Modular Symbols with Values in Beilinson-Kato Distributions}

\author{Cecilia Busuioc\footnote{\texttt{celiabusuioc@ucl.ac.uk}, Department of Mathematics, 
University College London, 
25 Gordon Street, London, WC1H 0AY}, Jeehoon Park\footnote{ \texttt{jpark.math@gmail.com}, QSMS, Seoul National University, 1 Gwanak-ro, Gwanak-gu, Seoul 08826, South Korea}, Owen Patashnick\footnote{ \texttt{owen.patashnick@kcl.ac.uk}, Department of Mathematics, King's College London, London, UK; School of Mathematics, University of Bristol, Bristol, UK; The Heilbronn Institute for Mathematical Research, Bristol, UK}, and Glenn Stevens\footnote{ \texttt{ghs@math.bu.edu}, Department of Mathematics, Boston University, Boston, MA 02215}}

\date{\today}

\maketitle
\begin{abstract}
		For each integer $n\geq 1$, we construct a $\GL_n(\Q)$-invariant modular symbol $\bm\xi_n$ with coefficients in a space of distributions that takes values in the Milnor $K_n$-group of the modular function field.
	The Siegel distribution $\bm\mu$ on $\Q^2$, with values in the modular function field, serves as the building block for $\bm\xi_n$; we define $\bm\xi_n$ essentially by taking the $n$-Steinberg product of $\bm\mu$.  
	The most non-trivial part of this construction is the cocycle property of $\bm\xi_n$; we prove it by using an induction on $n$ based on the first two cases $\bm\xi_1$ and $\bm\xi_2$; the first case is trivial, and the second case essentially follows from the fact that Beilinson-Kato elements in the Milnor $K_2$-group modulo torsion satisfy the Manin relations.

\end{abstract}

\tableofcontents

\section{Introduction}


The aim of this note is to 
define modular symbols with coefficients in $K$-theoretic-valued distributions.  Although the various elements that go into this definition appear in various places in the literature, to our knowledge no-one has yet packaged them in this way.  
In future work we hope to show  that this point of view may be useful for describing new and known Euler systems i.e. ``norm-compatible systems of cohomology classes" 
(especially that of Kato \cite{Kat}) via the distribution property of our modular symbols, and relating them to computing special values of $L$-functions and zeta functions.

The Siegel units on the upper half plane studied in \cite{Kat} by K. Kato play a key role in his construction of an Euler system for elliptic curves; see also \cite{Sch} for related Siegel-Kato functions on the universal elliptic curve.
Two key features of the Siegel units are known, namely,

(S-1) the Siegel units satisfy a  distribution property; see Kubert-Lang \cite{KubertLang} for details,

(S-2) the Siegel units satisfy the Manin relations in the Milnor $K_2$-group of a modular curve; 
see Brunault \cite{Bru} and (independently) Goncharov \cite{Gon}.

This distribution property and the Manin relations suggest the  existence of the aforementioned $\GL_2(\Q)$-invariant modular symbol. The goal of this article is to construct such a modular symbol explicitly.

We now recall the ``classical" notion of modular symbol over a subgroup $G$ of $\GL_2(\Q)$ taking values in a contravariant $G$-module $M$.   
For this we let $\Delta_0 :=\operatorname{Div}^0(\BP^1(\Q))$ which we endow with the action induced by the usual covariant action of $\GL_2(\Q)$ on  
$\BP^1(\Q)$.   An additive homomorphism 
$$
\xi : \Delta_0 \lra M
$$ 
 is said to be an $M$-valued modular symbol over $G$ if it satisfies the condition 
 $$
 \gamma^\ast \cdot \xi(\gamma D) = \xi(D)
 $$ 
 for all $D\in \Delta_0$ and all $\gamma\in G$.  We will also call this a $G$-invariant modular symbol, with the understanding that $G$ acts on $\Hom_\Z(\Delta_0, M)$ on the right by $(\xi,\gamma) \longmapsto \xi | \gamma$ where $\xi|\gamma \in \Hom_\Z(\Delta_0, M)$ is given by
$$
(\xi | \gamma)(D) = \gamma^\ast \cdot \xi ( \gamma D).
$$
Thus $\xi$ is said to be a modular symbol over $G$ if and only if it is fixed by $G$ with the above-defined right (i.e.\ contravariant) action.
Finally we will denote the group of all $M$-valued modular symbols over $G$ by $\operatorname{Symb}_G(M)$.

\noindent

For an integer $N\geq 1$, let $\Gamma(N)$ denote the modular congruence subgroup of full level $N$ structure and $\fH/\Gamma(N)$ the associated modular curve. 
It is well known that there is a smooth affine algebraic curve $Y(N)$ over $\C$ such that the associated analytic space $Y(N)(\C)$ is isomorphic to $\fH/\Gamma(N)$; see Appendix \ref{appendix} for details. Let $k_N=\C(Y(N))$ be the field of rational functions on $Y(N)$. 
For each $n \geq 1$, let $K_n^M(k_N)$ be the $n$th Milnor $K$-group of $k_N$. Let $k_\infty$ be the union of the $k_N$'s and let $\tilde K^M_2(k_\infty)$ be the associated direct limit of the groups
\be
\tilde K^M_2(k_N):=K^M_2(k_N) \otimes_\Z \Q.
\ee
We sometimes use the notation $ r \{x,y\}= \{x,y\} \otimes_\Z r$ where $r \in \Q$ and $x,y \in k_\infty^\times$.

%
%

Using this language, our main result is the construction of an explicit modular symbol
\be
{\bm\xi} \in \operatorname{Symb}_{\GL_2(\Q)} \biggl(\Dist(M_{2\times2}(\Q), \tilde K^M_2(k_\infty))\biggl)
\ee
using Siegel units.
Here $\Dist(M_{2\times2}(\Q), \tilde K^M_2(k_\infty))$ is the space of $\tilde K^M_2(k_\infty)$-valued distributions on the space $M_{2\times2}(\Q)$ of $2\times 2$ matrices with coefficients in $\Q$.


There is an important technical point regarding (S-2). For our construction of ${\bm\xi}$, we need the Manin relations for all $N\geq 1$.  Unfortunately \cite[Theorem 1.4]{Bru} only proves them for $N$ not divisible by $3$, and we find that Goncharov's claimed proof for all $N$ in \cite{Gon} has a gap.  (See Section \ref{GonchMethod} for more details)

Our key technical contribution is to salvage his argument and  actually construct symbols of theta functions on Milnor $K_3$ of universal elliptic curves whose images under the appropriate residue map give the Manin relations by Suslin's reciprocity.

We should also point out that \cite{Bru2} contains a proof different from ours which works for $\GL_2(\Q)$ uniformly for all levels $N$.

Finally, we also provide a generalization of ${\bm\xi}$ to a $\GL_n(\Q)$-setting, using the language of Ash-Rudolph \cite{AR} for $\GL_n(\Q)$-invariant modular symbols. More precisely, for each $n \geq 1$, we explicitly construct a $\GL_n(\Q)$-invariant modular symbol ${\bm\xi}_n$ with values in $\Dist (M_{n\times2}(\Q), \tilde K^M_n(k_\infty))$, which equals ${\bm\xi}$ when $n=2$. Here $M_{n\times2}(\Q)$ is the vector space of $n\times 2$ matrices with coefficients in $\Q$.

In the recent paper \cite{SV}, Sharifi and Venkatesh use motivic complexes (with theta functions similar to our article) to construct a $\GL_2(\Z)$-group cocycle with values in $K_2$-group of a modular curve. This construction is  closely related to our $\GL_2(\Q)$-invariant modular symbol $\bm\xi$ as discussed in Theorem \ref{modularsymbol}. Peter Xu generalized the method of \cite{SV} to a $\GL_n(\Z)$-setting: he sketched a construction of a variation of a $\GL_n(\Z)$-version of \cite{SV}, which he calls $\Theta_{\cC}^M$ in \cite[Section 4.2.2]{Xu}. 
It would be a good project to compare this generalization with our work.


 Bergeron, Charollois, and Garcia \cite{BCG} interpreted a construction of Eisenstein cohomology classes of $\GL_n$ as a transgression of the Euler class of a real vector bundle with $\SL_n(\Z)$-action. Peter Xu worked out in \cite{Xu} that ``a regulator" of his $\Theta_{\cC}^M$ is essentially the same as the Eisenstein cohomology class of \cite{BCG}.
 

Also note that  in \cite{FK}, Fukaya and Kato constructed ``zeta maps" (including $p$-adic \'etale regulators) to $K_2$ of modular curves and computed the image of the the Steinberg symbol of Siegel units under these maps,
which was a crucial part in their work on Sharifi's conjecture.




Under the well known isomorphisms $K_n^M(F) \cong \mathrm{CH}^n(F, n)$ for $F$ a field (see for example \cite[p. 177]{Totaro}), we have actually  constructed modular symbols with values in the non-torsion parts of certain motivic cohomology groups. We think it would be 
interesting to explore the putative existence of, and the putative properties of, modular symbols with values in more general motivic cohomology groups, 
 as well as relating them via  realizations to other well-known modular symbols with values in interesting $\Gamma$-modules (for $\Gamma$ a discrete subgroup of $\GL_n(\Q)$ which fixes a lattice in $\Q^n$).  
\subsection{Statement of theorems and the Manin relations}


We start by defining Siegel units of level $N$ following Kato (\cite[subsection 1.9]{Kat}). Let $A(N):=\frac{1}{N}\Z^2/\Z^2$.

\begin{definition}\label{Siegelunits}\cite[Definition 1.1]{Bru}
For $\ba = (a_1 + \Z, a_2 +\Z)\in A(N)$ with $0\le a_1,a_2 < 1$, we define $g_\ba = 1$ if $\ba = {\bf 0}$ and otherwise define
$$
g_\ba(\tau) := q^{\frac{1}{2}B_2(a_1)} \prod_{n\ge 0} \left(1 - q^{n+a_1} e^{2\pi i a_2} \right) \prod_{n\ge 1} \left(1 - q^{n-a_1} e^{-2\pi i a_2} \right)
$$
where $q=e^{2 \pi i \tau}$ and $B_2(x)=x^2-x+\frac{1}{6}.$ We also use the notation $g_{a_1,a_2}=g_\ba$ later in Definition \ref{glnm}.
\end{definition}

Let $\cR_N$ be the subring of $k_N$ consisting of functions that are holomorphic on the upper half-plane and let
$$
\cR_\infty := \bigcup_N{\cR_N}.
$$
Then $g_\ba^{12N}$ is an element of $\cR_N^\times$ (cf. \cite{KubertLang} p. 182).
Therefore we can regard $g_\ba$ as a well-defined element ``$g_\ba^{12N} \otimes_\Z \frac{1}{12N}$" of $\tilde K^M_1(k_N):= K^M_1(k_N) \otimes_\Z \Q \simeq k_N^\times \otimes_\Z \Q$.

Let $\cS(\Q^2)$ be the abelian group of $\Z$-valued locally constant functions on $\Q^2$ with bounded support; see \eqref{testftn} for details.
We let $\Q^\times$ act on $\cS(\Q^2)$ and $k_\infty$ through the action of $\GL_2(\Q)$ via the diagonal embedding of $\Q^\times$ in $\GL_2(\Q)$. 
Thus $\Q^\times$ acts trivially on $k_\infty$ and it acts on $\cS(\Q^2)$ as follows:
\be
(t^* f)(x)=f(tx), \quad t \in \Q^\times, x \in \Q^2.
\ee

The key feature of the Siegel units is the distribution property:
for all $0\neq t \in \Z$ and all $\ba \in A(N)$ we have
\be
g_\ba= \sum_{\substack{\bb \in A(N)\\ t \bb=\ba}} g_\bb,
\ee
which can be reformulated as the following proposition.
\begin{proposition}[The Siegel distribution]\cite[Theorem 1.8]{Colmez}\label{Sieg}
There is a unique $\Q^\times$-invariant (i.e. $t \in \Q^\times$, $\bm\mu(t^* f)=\bm\mu(f)$) homomorphism\footnote{This $\Q^\times$-invariance plays a role in defining $\GL_2(\Q)$-invariant modular symbols: see the first sentence of the proof of Proposition \ref{wdefined}.}
$$
{\bm\mu} : \cS(\Q^2) \lra \tilde K^M_1(k_\infty) 
$$
for which 
$$
{\bm\mu}([\ba+\Z^2]) = g_{\ba}\ \hbox{\rm for all ${\bf a}\in \Q^2/\Z^2$}
$$
where $[\ba+\Z^2]$ is the characteristic function on $\bf a +\Z^2$.
\end{proposition}

If $a_1, \cdots, a_n \in \Q^2$, let $A(a_1, \cdots, a_n) \in M_{n\times 2}(\Q)$ be the $n\times2$ matrix whose rows are $a_1, \cdots, a_n$.
In the following we will identify $\cS(\Q^2)\otimes\cdots \otimes \cS(\Q^2)$ with  $\cS(M_{n\times 2}(\Q))$ by identifying $\phi_1\otimes \cdots \otimes \phi_n\in \cS(\Q^2)\otimes \cdots \otimes\cS(\Q^2)$ with the test function on $M_{n\times 2}(\Q)$ defined by
\bea \label{identify}
(\phi_1\otimes \cdots \otimes \phi_n) : A(a_1,\cdots, a_n)\longmapsto \phi_1(a_1)\cdots \phi_n(a_n)\quad\hbox{\rm for all $a_i\in \Q^2$}.
\eea
For $\delta\in \Dist(M_{n\times 2}(\Q),\tilde K^M_n(k_\infty))$ and $\alpha\in \GL_n(\Q)$, we define the group action 
\be
\alpha^*\delta \in \Dist(M_{n \times 2}(\Q),\tilde K^M_n(k_\infty)), \quad 
(\alpha^* \delta)(\phi) = \delta((\alpha^{-1})^{*}\phi)
\ee
for an arbitrary element $\phi\in \cS(M_{n\times 2}(\Q))$, where $\alpha^*(\phi)(A) = \phi(\alpha A)$ 
for any $A\in M_{n\times 2}(\Q)$.

\begin{definition}\cite[1.4.2]{Colmez}\label{BKdist}
We define the Beilinson-Kato distribution 
$$
{{\bm\mu}^2} \in \Dist(M_{2\times2}(\Q),\tilde K^M_2(k_\infty)), \quad
{{\bm\mu}^2}(\phi_1\otimes \phi_2) := \left\{\,{\bm\mu}(\phi_1),{\bm\mu}(\phi_2)\,\right\}.
$$
\end{definition}

\begin{theorem}\label{modularsymbol}
There is a unique $\GL_2(\Q)$-invariant modular symbol
$$
{\bm\xi} : \Delta_0 \lra \Dist(M_{2\times2}(\Q),\tilde K^M_2(k_\infty))
$$
with the property that 
$${\bm\xi}\biggl((0)-(\infty)\biggr) = {{\bm\mu}^2}. $$
In other words, $\bm\xi$ is a group homomorphism satisfying
\be
\bm\xi (\g D) = (\g^{-1})^* \bm\xi(D), \quad D \in \Delta_0, \g \in \GL_2(\Q).
\ee
\end{theorem}


For any triple $r,s,t\in \PP^1(\Q)$ of distinct cusps, we need to show that the Manin relations hold
\be
&&{\bm\xi}(r\to s) + {\bm\xi}(s\to r) =0 \\
&&{\bm\xi}(r\to s) + {\bm\xi}(s\to t) + {\bm\xi}(t\to r) = 0.
\ee
The first relation follows immediately from the antisymmetry of the Steinberg symbol. 
However, the proof of the second Manin relation is more involved.
We prove the second Manin relation \eqref{Maninrelation} for Siegel units by applying Suslin's  reciprocity law to certain elements in $K_3$ of the function field of an elliptic surface.

We define the complex manifold $\cE_N:=(\C \times \fH)/ \tilde \Gamma(N)$ and a natural quotient map
\bea \label{univell}
\pi_N: \cE_N:=(\C \times \fH)/ \tilde \Gamma(N) \to Y_N:=\fH/\Gamma(N)
\eea
where the action of $\tilde \Gamma(N)=\Z^2 \rtimes \Gamma(N)$ on $\C \times \fH$ is the one described in the proof of Proposition \ref{Theta:transform} (see displayed equation \eqref{saction})
and $\pi_N$ is induced from the natural projection from $\C \times \fH$ to $\fH$. 
For $N\geq 3$, it is known that there is an algebraic variety $\cE(N)$ over $Y(N)$ of relative dimension $1$ such that the associated analytic space $\cE(N)(\C)$ is isomorphic to $\cE_N$ as complex manifolds; in fact, there is an elliptic surface $\overline \cE(N)$ over $X(N)$, which is the smooth compactification of $Y(N)$, in the sense of \cite[Definition on p. 202]{Silverman}; see Appendix \ref{appendix} for details.
Let $E(N)$ be the generic fiber of $\overline \cE(N)$, i.e. $E(N)=\Spec(k_N)\times_{X(N)}\overline \cE(N)$:
\[\begin{tikzcd}
	{E(N)} & {\overline \cE(N)} \\
	{\Spec(k_N)} & {X(N).}
	\arrow[" ", from=1-1, to=1-2]
	\arrow["{}"', from=1-1, to=2-1]
	\arrow[from=1-2, to=2-2]
	\arrow[" ", from=2-1, to=2-2]
\end{tikzcd}\]
Note that $E(N)$ is an elliptic curve over $k_N$ and the group of $E(N)(k_N)$ of $k_N$-rational points of $E(N)$ is exactly the group of points of order $N$ of $E(N)$ for $N\geq 3$; see \cite[Remark 5.5, Theorem 5.5]{Shioda} for more details.
\begin{definition} \label{fn}
For $N\geq 3$, let $F_N=k_N(E(N))$ be the field of rational functions on $E(N)$ over the field $k_N$. 
\end{definition}

Let $F_\infty:= \bigcup_{N} F_N$. Then it is easy to see that $F_\infty$ is a field extension of $k_\infty$ of transcendence degree 1, since $\cE(N)$ has relative dimension 1 over $Y(N)$.
We will use Theorem \eqref{SRL} (Suslin's reciprocity law for Milnor $K$-groups) applied to $$\Res: \tilde K^M_3(F_\infty) \to \tilde K^M_2(k_\infty)$$ 
as follows:

%
For each integer $N \geq 3$, we construct an element $\Phi_N \in K^M_3(F_N)$ such that
\bea
\Res(\Phi_N)=N^2(12N^2)^3\left( \{g_{\ba-\bb},g_{\bc-\ba}\} +  \{g_{\bc-\ba},g_{\bb-\bc}\} +  \{g_{\bb-\bc},g_{\ba-\bb}\} \right)
\eea
where $\ba,\bb,\bc \in A(N)=\frac{1}{N}\Z^2/\Z^2$;  
$$
\Phi_N:=\sum_{\bx \in A(N)} \left\{\frac{_{N}\!\Theta_{\ba}(u,\tau)^{12}}{_{N}\!\Theta_{\bx}(u,\tau)^{12}},\frac{_{N}\!\Theta_{\bb}(u,\tau)^{12}}{_{N}\!\Theta_{\bx}(u,\tau)^{12}},\frac{_{N}\!\Theta_{\bc}(u,\tau)^{12}}{_{N}\!\Theta_{\bx}(u,\tau)^{12}}\right\}.
$$
where $_{N}\!\Theta_{*}(u,\tau)^{12}$ are specific theta functions on universal elliptic curves which will be defined in the body of the paper;
see Lemma \ref{bcomputation} and Proposition \ref{klem} for further details.
By applying Suslin's reciprocity law to $\Res(\Phi_N)$, we get the following identity for all $N\geq 1$ (cf. Theorem \ref{sres}):
\bea\label{Maninrelation}
\{g_\ba, g_\bb \}+\{g_\bb, g_\bc\}+\{g_\bc, g_\ba\}=0\quad\text{in}\quad  \tilde K^M_2(k_N),
\eea
for $\ba,\bb,\bc \in A(N)$ satisfying $ \ba+\bb+\bc=0$.

For the $\GL_n(\Q)$-invariant modular symbol for $n\geq 1$, we use the language of Ash-Rudolph \cite{AS}: see Definition \ref{AR} for a precise definition.
We define the distribution ${\bm\mu}^n$ in the $\GL_n(\Q)$-setting.
\begin{definition}\label{glnSD}
The distribution ${\bm\mu}^n \in \Dist(M_{n\times2}(\Q),\tilde K^M_n(k_\infty))$ is defined by
\be
{\bm\mu}^n(\phi_1 \otimes \cdots \otimes \phi_n) = \{{\bm\mu}(\phi_1), \ldots, {\bm\mu}(\phi_n) \},
\quad \phi_i \in \cS(\Q^2),
\ee
where $\phi_1 \otimes \cdots \otimes \phi_n \in \cS(\Q^2) \otimes \cdots \otimes \cS(\Q^2) \simeq \cS(M_{n\times 2}(\Q))$.
\end{definition}
Note that ${\bm\mu}^1$ is exactly same as the Siegel distribution ${\bm\mu}$ in Proposition \ref{Sieg} and ${\bm\mu}^2$ is the Beilinson-Kato distribution in Definition \ref{BKdist}.
Let $\{e_1, \cdots, e_n\}$ be the standard basis of $\Q^n$ and $\{e_1^*, \cdots, e_n^*\}$ denote the dual basis of $(\Q^n)^*:=\Hom(\Q^n,\Q)$.

\begin{theorem}\label{glnmodularsymbol}
For each $n \geq 1$ there is a unique $\GL_n(\Q)$-invariant modular symbol ${\bm\xi}_n$
\bea
{\bm\xi}_n: (\operatorname{Hom}(\Q^n,\Q))^n \to \Dist(M_{n\times 2}(\Q), \tilde K^M_n(k_\infty))
\eea
such that ${\bm\xi}_n (e_1^*, \cdots, e_n^*) = {\bm\mu}^n$.
\end{theorem}
Note that ${\bm\xi}_1$ reduces to the Siegel distribution ${\bm\mu}$, and 
${\bm\xi}_2$ reduces to the Beilinson-Kato modular symbol ${\bm\xi}: \operatorname{Div}^0(\BP^1(\Q)) \to \Dist(M_{2\times2}(\Q), \tilde K^M_2(k_\infty))$: see Proposition \ref{cgcom} for details.
In Section \ref{glngeneral}, we will actually prove a basis-free version of Theorem \ref{glnmodularsymbol}, as we work with a $\Q$-vector space of dimension $n$ instead of $\Q^n$.

\vspace{1em}
We are aware that our result is not the most general possible.  We expect that our results may extend to curves over more general rings.
Note in particular that (almost) integrality is clear from our explicit computations of the residue map on Steinberg symbols of theta functions; for example, Lemma \ref{bcomputation} shows the precise denominator $\frac{1}{(12N^2)^3}$ in the computation. 
It seems possible to show that the computation generalizes to arbitrary moduli spaces of elliptic curves over a base scheme $S$ and their associated modular curves.
A potential sticking point is a precise generalization of Suslin's reciprocity result, though we feel confident that more recent advances in the field should be applicable (see for example \cite{Kriz}).\
It would certainly be a worthwhile project to pursue this abstraction.  However, we felt that a paper that explains the complex analytic case clearly and correctly would be a valuable prior contribution to the literature.

\vspace{1em}
The structure of the rest of the paper is as follows: 
Following a brief digression (Section \ref{GonchMethod}) on the gap in the proof in \cite{Gon}, we begin in Section \ref{background} by recalling the concepts of distributions, Milnor $K$-theory, and Suslin reciprocity.  In Section \ref{BeilinsonKato}, we continue by recalling theta functions and how to construct elliptic functions out of theta functions.  We then prove that the Manin relations are satisfied and complete the construction of the $\GL_2(\Q)$-invariant modular symbol.  In Section \ref{glngeneral}, we generalize our construction of modular symbols to those with coefficients in $\tilde K^M_n(k_\infty)$.  

\subsection{Remark on Goncharov's method}\label{GonchMethod}
In this section we briefly explain the argument of Goncharov in \cite{Gon}, we point out where we believe a gap exists in his argument, and we explain how we clarify it.

Let $E$ be an elliptic curve over an algebraically closed field $k$.
Let $k(E)$ be a function field of $E/k$.
For any field $F$, Goncharov considers the Bloch complex (\cite[section 2.1]{Gon})
\be
\cB_2(F) \xrightarrow{\delta_2} \wedge^2 F^\times
\ee
and the commutative diagram
\[\begin{tikzcd}
	{\cB_2(k(E))\otimes k(E)^\times} & {\wedge^3 k(E)^\times} \\
	{\cB_2(k)} & {\wedge^2k^\times.}
	\arrow["{\delta_2\wedge \mathrm{Id}}", from=1-1, to=1-2]
	\arrow["{\mathrm{Res}}", from=1-1, to=2-1]
	\arrow["{\mathrm{Res}}", from=1-2, to=2-2]
	\arrow["{\delta_2}", from=2-1, to=2-2]
\end{tikzcd}\]

Goncharov \cite[Theorem 2.4]{Gon} proved (in a previous paper) that there exists a map $h$ such that the following diagram commutes:
\[\begin{tikzcd}
	{\cB_2(k(E))\otimes k(E)^\times} & {\wedge^3k(E)^\times} \\
	{\cB_2(k)} & {\wedge^2k^\times.}
	\arrow["{\delta_2\wedge \mathrm{Id}}", from=1-1, to=1-2]
	\arrow["{\mathrm{Res}}", from=1-1, to=2-1]
	\arrow["h", from=1-2, to=2-1]
	\arrow["{\mathrm{Res}}", from=1-2, to=2-2]
	\arrow["{\delta_2}", from=2-1, to=2-2]
\end{tikzcd}\]
For each $N$-torsion point $a$ in $E(k)$ Goncharov defines an element (\cite[Corollary 2.6]{Gon})
\be
\theta_E(a) \in k^\times \otimes \Q.
\ee
Then, for any $N$-torsion points $a,b$ on $E(k)$, he claims that the function
\be
t \mapsto \frac{\theta_E(t-a)^N}{\theta_E(t-b)^N}
\ee
is an element of $k(E)^\times$ satisfying $\mathrm{Div} (\frac{\theta_E(t-a)^N}{\theta_E(t-b)^N}) = N (\{a\} - \{b\})$.
(see \cite[section 2.4, page 1889-1890]{Gon}).  
Note that Goncharov is claiming something quite strong here, namely that he has constructed a function on $E$ with a very specific form.  Although it is true that the divisor $N (\{a\} - \{b\})$ on an elliptic curve over a field is principal, there is no sense in which such a function is canonical or has a canonical form (see for example the opening sentence of \cite[section 1.2]{Sch}).  Furthermore, it is no longer true in general that $N (\{a\} - \{b\})$ is principal on an elliptic curve over a base scheme $S$ different from a field. (For an explicit example see the last remark in \cite[section 1.2]{Sch}.)

Goncharov then computes the image of the wedge product of theta functions $\frac{\theta_E(t-a)^N}{\theta_E(t-b)^N}$ in $k(E)^\times$ under the map $h$ (in \cite[Theorem 2.4]{Gon}) in order to prove the Manin relations for Siegel units.  
So for Goncharov, the form of his alleged function is crucial to his proof.
The problem is that he does not actually define $\theta_E(t)$ for $t \in E(k)$ (Note that $\theta_E(t)$ is defined only for each $N$-torsion point $t$) in \cite{Gon} and thus the notation $\frac{\theta_E(t-a)^N}{\theta_E(t-b)^N}$ does not make sense as an element of $k(E)^\times$.
	
There is a hint as to what Goncharov may have had in mind here regarding the construction of $\theta_E(t)$ in a reference \cite{GL} to a previous paper with A. Levin.  However, this does not seem to help. In \cite[Remark, p 400]{GL} the authors specify the need to view $t\mapsto \theta_E(t)$ as valued in the Bloch group $B_2(E)$, and it is clear from its properties that the group $B_2(E)$ does not contain $k(E)^\times$ as a subgroup.


Our clarification has two important features that differentiate it from \cite{Gon}.  
The first is that we do not use the map $h$ (in \cite[Theorem 2.4]{Gon}), but rather a related result that follows from the Suslin reciprocity law for $\wedge^3 k(E)^\times \xrightarrow{\mathrm{Res}} \wedge^2 k^\times$, and the second is that we  explicitly (Proposition \ref{dcomputation}) find a function $f_{a,b}$ which satisfies $\mathrm{Div}(f_{a,b}) = 12N^2 (\{a\}-\{b\})$ using the Kato-Siegel functions (see for example \cite[Theorem 1.2.]{Sch} for details about how such functions are constructed over a more general base). 



Note that we do not see how to recover the function Goncharov wants from our construction; if one defines $\Theta_{\ba}$ in a similar way without $N$,
	 then $\Theta_{\ba}^{12}/\Theta_{\bb}^{12}$ would not be invariant under $\tilde \Gamma(N)$ (see the footnote after displayed equation \eqref{univell})  and so is not a meromorphic function.  In particular, $(\Theta_{\ba}^{12}/\Theta_{\bb}^{12})^N$ does not define a function with divisor $N (\{a\} - \{b\})$.


To be clear, our proof in Section \ref{ManinRelsProof} is similar in outline to Goncharov's plan in \cite{Gon} for the proof of the Manin relations.  In particular, we use (a corrected version of) Goncharov's beautiful trick (Lemma \ref{gonchvariantlemma}; see also \cite{Gon}, Lemma 2.11) of summing over all $N$-torsion points to get zero.

%
%
%

\subsection{Acknowledgements}
All four authors warmly thank three anonymous referees for their close reading of our paper and for their  substantive, perceptive, and constructive suggestions for improving it. 
 
The work of 
the second author was supported by the National Research Foundation of Korea (NRF-2021R1A2C1006696) and the National Research Foundation of Korea (NRF) grant funded by the Korea government (MSIT) (No.2020R1A5A1016126).  The work of the third author was supported by the Heilbronn Institute for Mathematical Research, Bristol, UK.

\section{Background and notation}\label{background}

\subsection{Test functions and distributions}
Throughout this paper, $V$ will be a finite dimensional $\Q$-vector space of dimension $n$ and $V^\ast := \Hom(V,\Q)$ will denote the dual space.
We let $G := \Aut_\Q(V)$ be the group of $\Q$-linear automorphisms of $V$.

By a lattice in $V$ we mean a finitely generated additive subgroup of $V$ that spans $V$ as a $\Q$-vector space.  A subset
$X\subseteq V$ is said to be {\it uniform} if there is a lattice $L$ in $V$ for which $X+\ell = X$ for every $\ell \in L$, in which case we say
$X$ is uniform with respect to $L$.  
We say $X$ is {\it bounded} if $X$ is contained in a lattice $M$ in $V$, in which case we say $X$ is bounded by $M$.   The collection of bounded uniform subsets of $V$ generates a topology on $V$ which we call the {\it lattice topology} on $V$.   

For a function $f$ on 
$V$ with bounded support and a vector $\ell\in V$ we say $\ell$ is a period of $f$ if $f(x+\ell) = f(x)$ for all $x\in V$.   The set of all periods of $f$ will be denoted 
$L_f$.   We say $f$ is locally constant if $L_f$ contains a lattice in $V$.  If $f\ne 0$, then both $L_f$ and the subgroup $M_f$ of $V$ generated 
by the support of $f$ are lattices in $V$.  We then call $L_f$ the period lattice of $f$ and $M_f$ the support lattice of $f$ and note that $L_f \subseteq M_f$.
We define
\bea \label{testftn}
\cS(V) := \{\, f : V \lra \Z\,|\, f \hbox{ has bounded support and is locally constant}\}.
\eea
We note that $G$ acts contravariantly on $\cS(V)$ by pull-back:  $(\g^\ast f) (x) = f( \g x)$.
For an arbitrary abelian group $A$, we define the group of $A$-valued distributions on
$V$ to be the group
$$
\Dist(V,A) := \Hom(\cS(V),A).
$$
For $\mu \in \Dist(V,A)$ and $f\in \cS(V)$ 
we sometimes write $\mu(f)$ in one or more of the following (suggestive) forms:
$$
\mu(f) =: \int f\cdot d\mu = \int f(x)d\mu(x) = \int_V f(x)d\mu(x).
$$
 In the applications $A$ will always be endowed with a contravariant action by a subgroup $H \subseteq G$.   In that case,  $\Dist(V,A)$ inherits a natural contravariant action by $H$, satisfying
\bea \label{distributionaction}
\int \gamma^\ast f \cdot d\gamma^\ast\mu = \gamma^\ast \left(\int f(x)d\mu(x)\right),
\eea
for any $\mu\in \Dist(V,A)$, $f\in \cS(V), x \in V$ and $ \gamma\in H$.

%

\subsection{Milnor $K$-groups and Suslin reciprocity}

We briefly recall the Milnor ring of a field.
Let $F$ be a field. The Milnor $K$-groups are defined as follows:
\begin{eqnarray*}
K^M_0(F) &=& \Z, \quad K^M_1(F) = F^\times \\
K^M_2(F) &=& F^\times \otimes F^\times / \langle a\otimes b :  a+b= 1 \rangle \\
K^M_n(F) &=& F^\times \otimes \cdots \otimes F^\times / \langle a_1 \otimes \cdots \otimes a_n
 : \text{ there exists } i \text { such that } a_i+a_{i+1}= 1 \rangle \\
 K^M_\ast(F) &=& \bigoplus_{n\geq 0} K^M_n(F).
\end{eqnarray*} 
Then $K^M_\ast(F)$ is a graded $\Z$-algebra which is a quotient of the tensor algebra on copies of $F^\times$ modulo the homogenous ideal in degree 2 generated by the Steinberg relation ($a_i+a_{i+1}=1$)\footnote{we will hereafter abuse notation and refer to this ideal as ``the Steinberg relations''}.
  We use the notation $\{x_1, \cdots, x_n\} \in K^M_n(F)$ to denote the element $x_1\otimes \cdots \otimes x_n$ modulo the Steinberg relations.  
The abelian group structure of $K^M_n(F)$ is determined by the following multi-linearity property:
$$
\{x_1, x_2, \ldots, x_n\} +\{x_1', x_2, \ldots, x_n  \}=\{x_1 \cdot x_1', x_2, \ldots, x_n\}.
$$
where similar multi-linearity holds for the other components. 

A straightforward computation (\cite[Lemma 1.1]{Sus}) shows that 
\begin{eqnarray*}
&&\{x,y\} = -\{y, x\}, \quad x,y \in A^\times, \\
&&\xi \cdot\eta = (-1)^{nm} \eta \cdot \xi, \quad \xi \in K^M_n(A), \eta \in K^M_m(A)\\
&&\{x_1, \ldots, x_n\}=0, \quad \text{if} \quad x_1+\cdots+x_n=0 \text{ or } 1. 
\end{eqnarray*}

We use the notation
\be
\tilde K^M_n(F):=K^M_n(F) \otimes_\Z \Q.
\ee

Let $v$ be a discrete valuation on $F$ and $\cO_v=\{x \in F : v(x) \geq 0\}$ be the corresponding valuation ring whose maximal ideal is denoted by $\frak m_v=\{x \in F : v(x) > 0\}$.
Let $k(v)=\cO_v/\frak m_v$ be the residue field.
For an element $a \in \cO_v$ we denote by $\bar a$ the canonical image of $a$ in $k(v)$.

\begin{lemma}\cite[Lemma 1.3]{Sus} \label{boundary}
There is an additive homomorphism $\partial_v$ from the additive group of the ring $K^M_\ast(F)$ to the additive group of $K^M_{\ast-1}(k(v))$ such that
\begin{enumerate}[(1)]
\item $\partial_v (\{u_1, \ldots, u_{n-1}, a\} = v(a) \cdot \{ \bar u_1, \ldots, \bar u_{n-1}\}$ where $u_1, \ldots, u_{n-1} \in \cO^\times$ and $a \in F^\times$.
\item $\partial_v$ is an epimorphism of graded groups of degree -1 whose kernel contains $\{1 + \frak m_v\}\cdot K^M_\ast (F)$.
\item the following diagrams are commutative:
\[\begin{tikzcd}
	{K^M_1(F)} & {K^M_0(k(v))} && {K^M_2(F)} & {K^M_1(k(v))} \\
	{F^\times} & \Z && {(\wedge^2 F^\times)/\sim} & {k(v)^\times}
	\arrow["{\partial_v}", from=1-1, to=1-2]
	\arrow["{\partial_v}", from=1-4, to=1-5]
	\arrow["v", from=2-1, to=2-2]
	\arrow["{( \ ,\ )_v}", from=2-4, to=2-5]
	\arrow["\simeq", from=2-1, to=1-1]
	\arrow["\simeq", from=1-2, to=2-2]
	\arrow["\simeq", from=2-4, to=1-4]
	\arrow["\simeq", from=1-5, to=2-5]
\end{tikzcd}\]
where $\sim$ denotes the Steinberg relations and $(a,b)_v =\bar c$, where $c=(-1)^{v(a)v(b)}\frac{a^{v(b)}}{b^{v(a)}}$.  
\end{enumerate}
\end{lemma}
(Note that we choose to define the tame symbol with a sign matching class field theory rather than the usual definition in classical descriptions of algebraic $K$-theory.  This choice of course has no bearing on the validity of subsequent constructions in this paper.)

Suppose that $F$ is a finite extension of a field $k$.
Then there is a well-defined norm map (see \cite[Lemma 1.8]{Sus} for details)
\be
N_{F/k}:K^M_n(F) \to K^M_n(k), \quad n=0,1,2.
\ee
In general, there is a well-defined norm map
\be
N_{F/k}: \tilde K^M_n(F) \to \tilde K^M_n(k), \quad n\geq 0.
\ee

\begin{theorem}[Suslin's reciprocity law]\cite[Theorem 2.6]{Sus}\label{SRL}
If $F$ is a field extension of a field $k$ with transcendence degree 1, then for each $x\in \tilde K^M_\ast(F)$
\be
\Res(x):=  \sum_{w \in \Sigma(F/k)} N_{k(w)/k}(\partial_w(x))=0 \ \ \text{(this is a finite sum),}\quad  \Res:\tilde K^M_n(F) \to \tilde K^M_{n-1}(k),
\ee
where $\Sigma(F/k)$ is the set of all discrete valuations (including archimedean valuations) of $F$ which are trivial on $k$. This holds for $x \in K^M_n(F)$ for $n=1,2,3$, instead of $\tilde K^M_n(F)$.
\end{theorem}

\section{The Beilinson-Kato modular symbol}\label{BeilinsonKato}

\subsection{Siegel units and theta functions}


We begin with the basic theta function $\Theta$ defined on the space $\tfH := \C\times \fH$ by
$$
\Theta(u,\tau) := q^{\frac{1}{12}}(t^{1/2} - t^{-1/2}) \prod_{n=1}^\infty \left(1 - q^nt\right) \left(1 - q^nt^{-1}\right)
$$
with $t = e^{2\pi i u}$ and $q=e^{2\pi i \tau}$.  The product converges uniformly on compact subsets of $\tfH$, hence defines a holomorphic function on that space.  One checks easily that, as a function of $u\in \C$, $\Theta(u,\tau)$ has simple zero at each $\omega\in \Lambda_\tau$ and no zeroes outside of $\Lambda_\tau$.  In particular, for fixed $\tau$, the divisor of $\Theta(u,\tau)$ as function of $u$ is the characteristic function $[\Lambda_\tau]$ of the lattice $\Lambda_\tau=\Z\tau+\Z$ and is therefore invariant under translation by elements of $\Lambda_\tau$.   

\begin{proposition}\label{Theta:transform}
	For all $\omega = r\tau + s \in \Lambda_\tau$ we have
	$$
	\frac{\Theta(u+\omega,\tau)}{\Theta(u,\tau)} = \psi(\omega)e^{-2\pi i r (u + \omega/2)} = (-1)^s(-t)^{-r}q^{-r^2/2}
	$$
	where 
	$$
	\psi(\omega) = \begin{cases}
		1&\hbox{\rm if $\omega\in 2\Lambda_\tau$;}\\
		\\
		-1&\hbox{\rm otherwise.}\\
	\end{cases}.
	$$ 
\end{proposition}

\begin{proof} 
	This is well-known \cite{Lang}.
\end{proof}

Let $\R^2$ denote the space of two-dimensional real row vectors and $\tilde G := \R^2 \rtimes \GL^+_2(\R)$ be the semidirect product group in which the operation $\ast$ is given by
$$
(\bx_1,\gamma_1)\ast (\bx_2,\gamma_2) = (\bx_1 + \bx_2\gamma_1^{-1}, \gamma_1\gamma_2)
$$
for any $\bx_i\in \R^2$ and $\gamma_i \in \GL^+_2(\R)$.   Setting $\tfH := \C \times \fH$, one easily checks\footnote{One needs the cocycle condition $j(\g,\tau)\cdot j(\g', \g\tau)=j(\g'\g, \tau)$ where $j(\g,\tau)=c\tau +d$ and $\g=\begin{pmatrix}a&b\\c&d\end{pmatrix}$.} that the map
\bea \label{saction}
\begin{array}{ccccc}
	\tilde G &\times& \tfH &\lra&\tfH \\
	(\bx,\gamma)&,& (u,\tau)&\longmapsto& \biggl(\displaystyle\frac{u+\rho_{\bx\gamma}(\tau)}{c\tau + d}, \gamma\tau\biggr)\\
\end{array}
\eea
defines a left action of $\tilde G$ on $\tfH$ for $\bx = (x_1,x_2)\in \R^2$, where $\rho_\bx(\tau)$ is defined by
$$
\rho_\bx(\tau) := x_1\tau + x_2.
$$ 
The group $\tilde G$ acts on the right by pullback of functions on $\tfH$.
For any function $f(u,\tau)$ on $\tfH$ and any $\tilde\gamma := (\bx,\gamma)\in \tilde G$ we have
$$
\tilde\gamma^\ast f(u,\tau) = f\left(\frac{u + \rho_{\bx\gamma}(\tau)}{c\tau +d},\gamma\tau\right)
$$
Finally, let $\Gamma := \SL_2(\Z)$ and define $\tilde\Gamma := \Z^2 \rtimes \Gamma \subseteq \tilde G$.   Then $\tilde\Gamma$
is easily seen to be a subgroup of $\tilde G$.   

%

\begin{definition} \label{mero}
	We define 
	$_{N}\!\Theta :\tilde \fH=\C \times \fH \lra \C$ by
	\bea \label{thetad}
	_{N}\!\Theta(u,\tau) :=  \Theta(u,\tau)^{N^2}/\Theta(Nu,\tau).
	\eea
Note that $_{N}\!\Theta$ is a meromorphic function on $\tilde \fH$, since the denominator $\Theta(Nu,\tau)$ does not introduce essential singularities; $\Theta$ is a holomorphic function on $\tilde \fH$ and $\Theta(u,\tau)$ has simple zero at each $\omega\in \Lambda_\tau$ and no zeroes outside of $\Lambda_\tau$. 

\end{definition}

\noindent
The main result of this section is the following theorem.

\begin{theorem}\label{vartheta:invariance}
	There is a multiplicative character $\epsilon : \tilde\Gamma \lra {\mathbb \mu}_{12}\subseteq \C^\times$ such that the function $_{N}\!\Theta$ satisfies the relation
	$$
	\tilde\gamma^\ast (_{N}\!\Theta) =\epsilon(\tilde\gamma)\cdot_{N}\!\Theta
	$$
	for all $\tilde\gamma \in \tilde\Gamma$.  More precisely, for $\bx\in \Z^2$,  $T=\begin{pmatrix}1&1\\0&1\end{pmatrix}$, and $S=\begin{pmatrix}0&1\\-1&0\end{pmatrix}$
	we have 
	$$
	\epsilon(\bx, 1) = 1,\ \  \epsilon(0, T) = e^{2\pi i(N^2-1)/12}\ \ \hbox{\rm and}\ \ \epsilon(0,S) = i^{N^2-1}.
	$$
\end{theorem}

\begin{proof}
	Since $(\bx, 1)$, $(0,T)$ and $(0,S)$ generate $\tilde\Gamma$ we need only prove the following four identities:
	\begin{itemize}
		\item[\rm (1)]  $_{N}\!\Theta\left(u + 1,\tau\right)= _{N}\!\Theta(u,\tau)$;
		\item[\rm (2)]  $_{N}\!\Theta\left(u + \tau ,\tau\right)= _{N}\!\Theta(u,\tau)$;
		\item[\rm (3)]  $_{N}\!\Theta\left(u,\tau+1\right)=e^{2\pi i (N^2-1)/12}\cdot_{N}\!\Theta(u,\tau)$; and
		\item[\rm (4)]  $_{N}\!\Theta\left(-\frac{u}{\tau},-\frac{1}{\tau} \right)= i^{N^2-1}\cdot_{N}\!\Theta(u,\tau)$.
	\end{itemize}
	
	To prove (1) and (2) we use Proposition {\ref{Theta:transform}}
	$$
	\Theta(u+r \tau + s,\tau) = (-1)^{r+s+rs}e^{-2\pi i r(u+\frac{r\tau+s}{2})}\cdot \Theta(u,\tau)
	$$
	from which we deduce 
	\begin{itemize}
		\item[\rm (i)] $\Theta(u+1,\tau) = -\Theta(u,\tau)$ and $\Theta(N(u+1),\tau) = (-1)^N\Theta(Nu,\tau)$; 
		\item[\rm (ii)] $\Theta(u+\tau,\tau) = -t^{-1}q^{-1/2}\Theta(u,\tau)$  and  $\Theta(N(u+\tau),\tau) = (-t)^{-N^2}q^{-N^2/2}\Theta(Nu,\tau)$.
	\end{itemize}
	Clearly (i) $\Longrightarrow$ (1), and (ii) $\Longrightarrow$ (2) and (3) follows from inspection of the $q$-expansions.  So only (4) remains to be proved.

	To prove (4) we consider the function
	$$
	\psi(u,\tau) = e^{\pi i u^2/\tau}\cdot\frac{\Theta(u,\tau)}{\Theta\left(-\frac{u}{\tau}, -\frac{1}{\tau}\right)}.
	$$
	A simple calculation using (1) and (2) above shows that $\psi(u+1,\tau) = \psi(u+\tau,\tau)=\psi(u,\tau)$.  Thus $\psi(u,\tau)$ is an elliptic function for the lattice $\Lambda_\tau$.
	But since the numerator and denominator of the above expression for $\psi(u,\tau)$ vanish nowhere outside of $\Lambda_\tau$ and they both have a simple zero at the origin, it follows that
	$\psi(u,\tau)$ has no zeroes or poles in the complex plane. Thus $\psi(u,\tau)$ is a non-zero constant (depending on $\tau$).   To evaluate this constant we calculate $\psi(u,\tau)$ at $u=0$.
	We have
	$$
	\psi(0,\tau) = \displaystyle  \lim_{u\to 0}\frac{\Theta(u,\tau)}{\Theta(-\tfrac{u}{\tau},-\tfrac{1}{\tau})} =\displaystyle -\tau \frac{\Theta^\prime(0,\tau)}{\Theta^\prime(0,-\tfrac{1}{\tau})}
	$$
	Writing 
	$\Theta(u,\tau)	= 2i\sin(\pi u)T(u,\tau)$ where $T(u,\tau)=q^{1/12}\prod_{n>0}(1-q^nt)(1-q^nt^{-1})$, 
	we note that 
	$$
	\Theta^\prime(u,\tau) = 2\pi i\cos(\pi u) \cdot T(u,\tau)  + 2i\sin(\pi u)\cdot T^\prime(u,\tau). 
	$$
	Setting $u=0$ into this identity gives us $\Theta^\prime(0,\tau) = 2\pi i T(0,\tau)$.
	But clearly
	$$
	T(0,\tau) = \eta(\tau)^2
	$$
	where $\eta(\tau)$ is the famous Dedekind $\eta$-function.   Thus
	$$
	\Theta^\prime(0,\tau) = 2\pi i \cdot \eta(\tau)^2  
	$$
	and replacing $\tau$ in this identity by $-1/\tau$ we obtain
	$$
	\Theta^\prime(0,-1/\tau) = 2\pi i \cdot \eta(-1/\tau)^2.
	$$
	It follows now that
	$$
	\psi(0,\tau) = -\tau\cdot\frac{\eta(\tau)^2}{\eta(-1/\tau)^2}
	$$
	and from the well-known identity
	$\eta(\tau)^2 = \frac{i}{\tau} \cdot\eta(-1/\tau)^2$
	we conclude 
	$$
	\psi(0,\tau) = -i.
	$$
	
	We therefore have
	$$
	\Theta\left(-\frac{u}{\tau}, -\frac{1}{\tau}\right) = i e^{2\pi i u^2/\tau} \cdot \Theta(u,\tau)
	$$
	from which it follows at once that
	$$
	_{N}\!\Theta\left(-\frac{u}{\tau}, -\frac{1}{\tau}\right) = \frac{\Theta\left(-\frac{u}{\tau}, -\frac{1}{\tau}\right)^{N^2}} {\Theta\left(-\frac{Nu}{\tau}, -\frac{1}{\tau}\right)} 
	= i^{N^2-1} \frac{\Theta(u,\tau)^{N^2}}{\Theta(Nu,\tau)}  = {i^{N^2-1}} _{N}\!\Theta(u,\tau)
	$$
	and (4) is established.  This completes the proof.
	
\end{proof}

The following corollary follows immediately from Theorem \ref{vartheta:invariance}. Recall that $A(N):=\frac{1}{N}\Z^2/\Z^2$.  
\medskip
\begin{corollary}\label{periodic}
	For $\ba=(a_1+\Z,a_2+\Z) \in A(N)-(0,0)$, define $_{N}\!\Theta_{\ba}: \tfH := \C\times \fH \to \C$ as follows:
	\bea\label{cdiv}
	_{N}\!\Theta_{\ba}(u,\tau):=_{N}\!\Theta (u-a_1\tau-a_2, \tau):=\frac{\Theta(u-a_1\tau-a_2,\tau)^{N^2}}{\Theta(Nu-Na_1\tau-Na_2,\tau)}.
	\eea
	
	(1)
	If $N >0$ is prime to 6, then $_N\Theta(u,\tau)$ is invariant under $\tilde\Gamma=\Z^2 \rtimes \SL_2(\Z)$. In particular $_{N}\!\Theta$ is periodic in $u$ with period $\Lambda_\tau$.
	
	(2) For any positive integer $N$, the function $_N\Theta(u,\tau)^{12}$ is invariant under $\tilde\Gamma=\Z^2 \rtimes \SL_2(\Z)$.
	
	(3) 
	The function $_{N}\!\Theta_{\ba}(u,\tau)^{12}$ is invariant
	under $\tilde \Gamma(N)=\Z^2 \rtimes \Gamma(N)$.
	
\end{corollary}
\begin{proof}
	(1) Since $N >0$ is prime to 6, $N=6k+1$ or $N=6k+5$ for some integer $k$. Then
	\be
	\epsilon(0, T) = e^{2\pi i({N}^2-1)/12}=1\ \ \hbox{\rm and}\ \ \epsilon(0,S) = i^{{N}^2-1}=1.
	\ee
	Thus the result follows.
	
	(2) It is clear, since $\epsilon$ has values in the multiplicative group 
	of $12th$ roots of unity.

	%
	%

	(3) First, we have
	\be
	_{N}\!\Theta_{\ba} (u + r\tau +s ,\tau)^{12} = _{N}\!\Theta (u-a_1\tau-a_2+ r\tau +s , \tau)^{12}=_{N}\!\Theta_{\ba} (u,\tau)^{12}, \quad r\tau +s \in \Lambda_\tau=\Z\tau+\Z.
	\ee
	Second, we have
	\be
	&&_{N}\!\Theta_\ba \left(\frac{u}{c\tau+d}, \frac{a\tau+b}{c\tau+d}\right)^{12} \\
		&=& _{N}\!\Theta \left(\frac{u}{c\tau+d}-a_1\cdot \frac{a\tau+b}{c\tau+d}-a_2, \frac{a\tau+b}{c\tau+d}\right)^{12}
	\quad (a=1+Nk_1, b=Nk_2, c=Nk_3, d=1+Nk_4)\\
	&=&   _{N}\!\Theta \left(\frac{u-a_1\tau-a_2}{c\tau+d}-\frac{a_1Nk_1\tau+a_2Nk_4}{c\tau+d}-\frac{a_2Nk_3 \tau + a_1Nk_2}{c\tau+d}, \frac{a\tau+b}{c\tau+d}\right)^{12} \\
         &=&  _{N}\!\Theta (u-a_1\tau-a_2 - a_1Nk_1\tau-a_2Nk_4-a_2Nk_3 \tau - a_1Nk_2, \tau)^{12} \\
	&=&   	_{N}\!\Theta_{\ba} (u,\tau)^{12} \quad (a_1Nk_1, a_2Nk_4,a_2Nk_3,a_1Nk_2 \in \Z)
	\ee
	for $\gamma=\begin{pmatrix} a & b \\ c & d \end{pmatrix} \in \Gamma(N)$.
	This implies that $ _{N}\!\Theta_{\ba} (u,\tau)^{12}$ is invariant under $\tilde \Gamma(N)=\Z^2 \rtimes \Gamma(N)$, since the action of $\tilde \Gamma(N)=\Z^2 \rtimes \Gamma(N)$ is generated by the above two actions of $\Gamma(N)$ and $\Lambda_\tau$.
\end{proof}

The functions $_{N}\!\Theta_\ba$ (for $N$ prime to 6) and $_{N}\!\Theta^{12}_\ba$ are meromorphic on $\tilde \fH$ (as explained in Definition \ref{mero}) and invariant under $\tilde \Gamma$ (Corollary \ref{periodic}).  Thus, they are meromorphic functions on the quotient complex manifold $\tilde \fH/\tilde \Gamma$. Similarly,  $_{N}\!\Theta^{12}_\ba$ is a meromorphic function on $\cE_N:=(\C \times \fH)/ \tilde \Gamma(N)$. 
We recall the natural quotient map \eqref{univell}
\be
\pi_N: \cE_N:=(\C \times \fH)/ \tilde \Gamma(N) \to Y_N:=\fH/\Gamma(N)
\ee
and the elliptic surface $\overline \cE(N)$ over $X(N)$ for $N\geq 3$ in the introduction; also recall $\cE(N):=Y(N)\times_{X(N)}\overline \cE(N)$, which is an elliptic curve over $Y(N)$.
We show that $_{N}\!\Theta^{12}_\ba:\cE_N\simeq \cE(N)(\C) \to \C$ extends to a meromorphic function on the compact manifold $\overline \cE(N)(\C)$; see Proposition \ref{rationality} in Appendix \ref{appendix}.
By GAGA,
for $N\geq 3$, the function $_{N}\!\Theta^{12}_\ba$ is a rational function on $\overline \cE(N)$, i.e. $_{N}\!\Theta^{12}_\ba$ belongs to $F_N=k_N(E(N))$.   

Let $N\geq 3$. In order to study the divisors of $_{N}\!\Theta^{12}$ and $_{N}\!\Theta^{12}_\ba$, we define
\bea \label{divs}
D_\bx:=\{[(x_1\tau+x_2, \tau)]_{\tilde \Gamma(N)} \in \cE_N: \tau \in \fH \} \subset \cE_N.
\eea
for $\bx \in A(N)$. Note that $D_\bx$ is the image of the $N$-torsion section $s_\bx$ for the map $\pi_N:\cE_N \to Y_N$ defined by $[\tau]_{\Gamma(N)} \mapsto s_\bx([\tau]_{\Gamma(N)}):=[(x_1\tau+x_2, \tau)]_{\tilde \Gamma(N)}$. Here $[\cdot ]_G$ means the $G$-coset.
 A crucial fact is that, according to the discussion above, \eqref{thetad}, and Corollary \ref{periodic}, $_{N}\!\Theta^{12}=\left( \Theta(u,\tau)^{N^2}/\Theta(Nu,\tau) \right)^{12}$ is a meromorphic function on $\cE_N$ with divisor given by
$$
12 \cdot \left( N^2\cdot D_{\mathbf{0}}  - \sum_{\bx\in A(N)} D_\bx\right)
$$
where $\mathbf{0} =(0+\Z,0+\Z) \in A(N)$, and thus $_{N}\!\Theta_{\ba}^{12}$  in \eqref{cdiv} is a meromorphic function on $\cE_N$ whose divisor is given by
\bea \label{divisors}
12 \cdot \left( N^2\cdot  D_\ba  - \sum_{\bx\in A(N)} {D_\bx}\right).
\eea.
 
 

\begin{proposition}\label{dcomputation}
	The function $(_{N}\!\Theta_{\ba}/_{N}\!\Theta_{\bb})^{12}$ defines a meromorphic function\footnote{If one defined $\Theta_{\ba}$ in a similar way without $N$, then $\Theta_{\ba}^{12}/\Theta_{\bb}^{12}$ would not be invariant under $\tilde \Gamma(N)$: for example, by Proposition \ref{Theta:transform}, we have
\be
\frac{\Theta_{\ba}(u+\tau,\tau)}{\Theta_{\bb}(u+\tau,\tau)}= \frac{e^{2\pi i (a_1 \tau +a_2 -\tau/2)}}{e^{2\pi i (b_1 \tau +b_2 -\tau/2)} }\frac{\Theta_{\ba}(u,\tau)}{\Theta_{\bb}(u,\tau)}\neq \frac{\Theta_{\ba}(u,\tau)}{\Theta_{\bb}(u,\tau)}.
\ee
}
on $\cE_N$ whose divisor is given by
	\bea\label{tdiv}
	(12N^2)\cdot \left( D_\ba - D_\bb \right)
	\eea
	where
	$D_\ba=\{[(a_1\tau+a_2, \tau)]_{\tilde \Gamma(N)} \in \cE_N: \tau \in \fH \} \subset \cE_N$ is defined in \eqref{divs}.
	
\end{proposition}

\begin{proof}
This is clear from Corollary \ref{periodic} and \eqref{divisors}.
\end{proof}

%

\begin{lemma}\label{Theta:siegel} For $\ba \in A(N)-(0,0)$ and $(a_1,a_2) \in \frac{1}{N}\Z^2$ the least non-negative representative of $\ba$, we have
	\begin{itemize}
		\item[\rm (a)]  $\displaystyle -e^{\pi i a_2} q^{\frac{a_1^2}{2}}\cdot \Theta\left(a_1\tau + a_2,\tau \right) = g_\ba(\tau);$ and
		\item[\rm (b)]  $\displaystyle \frac{\Theta(Nu + Na_1\tau+Na_2, \tau)}{ \Theta(Nu,\tau)} = (-1)^{Na_2} (-t^{N})^{-Na_1}q^{-\frac{N^2a_1^2}{2}}$, where $t=e^{2\pi iu}$ and $q=e^{2\pi i \tau}$.
	\end{itemize}
	
\end{lemma}

\begin{proof}  Both of these are straightforward calculations.   For (a) we have  
	$$
	\begin{array}{rcl}
		\Theta\left(a_1\tau + a_2,\tau \right) &=&\displaystyle q^{\frac{1}{12}}\left( e^{\pi i a_2}q^{\frac{a_1}{2}} - e^{-\pi i a_2}q^{\frac{-a_1}{2}}\right)
		\prod_{n\ge 1}\left( 1-q^{n+a_1}e^{2\pi i a_2}\right)\left( 1-q^{n-a_1}e^{-2\pi i a_2}\right)\\
		&=&\displaystyle -q^{\frac{1}{12}}e^{-\pi i a_2}q^{\frac{-a_1}{2}} \prod_{n\ge 0}\left( 1-q^{n+a_1}e^{2\pi i a_2}\right)\prod_{n\ge 1} \left( 1-q^{n-a_1}e^{-2\pi i a_2}\right)\\
		&=&\displaystyle -e^{-\pi i a_2}q^{-\frac{a_1^2}{2}}q^{\frac{1}{2}(a_1^2-a_1+\frac{1}{6})} \prod_{n\ge 0}\left( 1-q^{n+a_1}e^{2\pi i a_2}\right)\prod_{n\ge 1} \left( 1-q^{n-a_1}e^{-2\pi i a_2}\right)\\
		&=&-e^{-\pi i a_2}q^{-\frac{a_1^2}{2}}g_\ba
	\end{array}
	$$
	The proof of (b) is a simple substitution in Proposition \ref{Theta:transform} using the fact that $Na_1, Na_2\in \Z$. 
\end{proof}

\begin{proposition}\label{vartheta:siegelspecial}
For any non-zero $\ba,\bb \in A(N)$ the meromorphic function  $(_{N}\!\Theta_{\ba}/_{N}\!\Theta_{\bb})^{12}$ on $\cE_N$ is holomorphic on the $\mathbf{0}$-section $s_{\mathbf{0}}$, and its pullback to $Y_N$ along $s_{\mathbf{0}}$ is
$
 \left(g_\ba / g_\bb\right)^{12N^2}.
$
\end{proposition}

\begin{proof}  
	Recall the definition $_{N}\!\Theta(u,\tau) :=  \Theta(u,\tau)^{N^2}/\Theta(Nu,\tau)$.
	This proposition follows from the following computation: from lemma \ref{Theta:siegel} we have
	
	$$
	\begin{array}{rcl}
		\displaystyle \frac{_{N}\!\Theta(u+ a_1\tau + a_2,\tau)}{_{N}\!\Theta(u+b_1\tau+b_2,\tau)}\biggm|_{u=0} 
		&=& \displaystyle \frac{\Theta(a_1\tau + a_2,\tau)^{N^2}}{\Theta(b_1\tau +b_2,\tau)^{N^2}} \cdot \frac{\Theta(Nu+Nb_1\tau + Nb_2,\tau)}{\Theta(Nu + Na_1\tau +Na_2,\tau)}\biggm|_{u=0}\\
		&=& \displaystyle \left(\frac{e^{-\pi i a_2} q^{-\frac{a_1^2}{2}}g_\ba}{e^{-\pi i b_2} q^{-\frac{b_1^2}{2}}g_\bb}\right)^{N^2}\cdot \frac{\Theta(Nu+Nb_1\tau + Nb_2,\tau)}{\Theta(Nu + Na_1\tau +Na_2,\tau)}\biggm|_{u=0}\\
		&=& \displaystyle \left(\frac{e^{-\pi i a_2} q^{-\frac{a_1^2}{2}}g_\ba}{e^{-\pi i b_2} q^{-\frac{b_1^2}{2}}g_\bb}\right)^{N^2}\cdot 
		\frac{(-1)^{Nb_2} (-t^{N})^{-Nb_1}q^{-\frac{N^2b_1^2}{2}}}{(-1)^{Na_2} (-t^{N})^{-Na_1}q^{-\frac{N^2a_1^2}{2}}}\biggm|_{t=1}\\
		&=& \displaystyle (-1)^{N(a_1+a_2 + b_1+ b_2) + N^2(a_2+b_2)}\left(\frac{g_\ba}{g_\bb}\right)^{N^2}.\\
	\end{array}
	$$
\end{proof}
\begin{corollary}\label{TS}
	For $\ba=(a_1+\Z,a_2+\Z), \bb=(b_1+\Z,b_2+\Z) \in \frac{1}{N}\Z^2/\Z^2-(0,0)$ and $\bx=(x_1+\Z,x_2+\Z) \in \frac{1}{N}\Z^2/\Z^2-(0,0)$ which is different from $\ba,\bb$, we have
	\be
	\frac{_{N}\!\Theta ((x_1-a_1)\tau+(x_2-a_2), \tau)^{12}}{_{N}\!\Theta ((x_1-b_1)\tau+(x_2-b_2), \tau)^{12}} =\left(\frac{g_{\bx-\ba}(\tau)}{g_{\bx-\bb}(\tau)}\right)^{12N^2}.
	\ee
\end{corollary}


\subsection{The Manin relations}\label{ManinRelsProof}
For simplicity of notation, we define:
\be
\theta_{\ba}(u, \tau):=_{N}\!\Theta_{\ba}(u,\tau)^{12},\quad
_{N}\!\Theta_{\ba}(u,\tau)=_{N}\!\Theta (u-a_1\tau-a_2, \tau)=\frac{\Theta(u-a_1\tau-a_2,\tau)^{N^2}}{\Theta(Nu-Na_1\tau-Na_2,\tau)},
\ee
where $_{N}\!\Theta_{\ba}(u,\tau)^{12}$ is an element of $F_N$ ($N\geq 3$) by Corollary \ref{periodic}.  
For $\ba, \bb \in A(N)$, we have
\be
\frac{\theta_{\ba}}{\theta_\bb} \in F_N^\times,  \quad
g_{\ba} \in \tilde K^M_1(k_N).
\ee


Since $g_{-\ba}=-e^{2 \pi i a_2} g_{\ba}$, we have $\{g_{-\ba}, g_{\bb}\}\equiv \{g_{\ba}, g_{\bb}\}$ in $\tilde K^M_2(k_N)$.
For $\ba \in A(N)$, the divisor
$D_\ba$
provides a discrete valuation on $F_N$ which is trivial on $k_N$. So we use the notation $\partial_\ba$ instead of $\partial_{v_\ba}$ where $v_\ba$ is the associated discrete valuation to $\ba$.

\begin{lemma} \label{bcomputation}
	For $N\geq 3$, let $\ba,\bb,\bc,\bx \in A(N)-(0,0)$ be distinct non-zero elements. We have the following identities in $\tilde K^M_2(k_N)$:
	$$
	\begin{array}{rclcl} \displaystyle
		\frac{1}{(12N^2)^3} {\partial}_\ba \left( \left\{\frac{\theta_{\ba}}{\theta_\bx},\frac{\theta_{\bb}}{\theta_\bx},\frac{\theta_{\bc}}{\theta_\bx}\right\}\right)
		&=& \displaystyle \left\{ \frac{g_{\ba-\bb}}{g_{\ba-\bx}}, \frac{g_{\bc-\ba}}{g_{\bx-\ba}} \right\} &=& \{g_{\ba-\bb},g_{\bc-\ba}\} - \{g_{\ba - \bb}, g_{\bx-\ba}\}-\{g_{\ba-\bx},g_{\bc-\ba}\} ;\\
		\\
		\displaystyle \frac{1}{(12N^2)^3} {\partial}_\bb \left( \left\{\frac{\theta_{\ba}}{\theta_\bx},\frac{\theta_{\bb}}{\theta_\bx},\frac{\theta_{\bc}}{\theta_\bx}\right\}\right)
		&=& \displaystyle \left\{ \frac{g_{\bb-\bc}}{g_{\bb-\bx}}, \frac{g_{\ba-\bb}}{g_{\bx-\bb}} \right\}
		&=&  \{g_{\bb-\bc}, g_{\ba-\bb}\} - \{g_{\bb-\bc}, g_{\bx-\bb}\} - \{g_{\bb-\bx},g_{\ba-\bb}\} ;\\
		\\
		\displaystyle \frac{1}{(12N^2)^3} {\partial}_\bc \left( \left\{\frac{\theta_{\ba}}{\theta_\bx},\frac{\theta_{\bb}}{\theta_\bx},\frac{\theta_{\bc}}{\theta_\bx}\right\}\right)
		&=&\displaystyle\left\{ \frac{g_{\bc-\ba}}{g_{\bc-\bx}}, \frac{g_{\bb-\bc}}{g_{\bx-\bc}} \right\} 
		&=& \{g_{\bc-\ba},g_{\bb-\bc} \} - \{g_{\bc -\ba},g_{\bx-\bc}\} - \{g_{\bc - \bx}, g_{\bb-\bc}\};\\
		\\
		\displaystyle \frac{1}{(12N^2)^3} {\partial}_\bx \left( \left\{\frac{\theta_{\ba}}{\theta_\bx},\frac{\theta_{\bb}}{\theta_\bx},\frac{\theta_{\bc}}{\theta_\bx}\right\}\right) 
		&=&\displaystyle \left\{ \frac{g_{\bx-\bc}} {g_{\bx-\ba}}, \frac{g_{\bx-\bb}}{g_{\bx-\ba}}\right\}
		&=&-\left\{ g_{\bx-\bc},g_{\ba- \bx}\right\} - \{g_{\bx -\ba}, g_{\bb-\bx}\} - \{ g_{\bx -\bb}, g_{\bc - \bx}\} \\
	\end{array}
	$$
	where $\partial_*: \tilde K^M_3(F_N) \to \tilde K^M_2(k(*))$ is the boundary map in Lemma \ref{boundary}.
	Consequently, we have
	$$
	\frac{1}{(12N^2)^3} {\Res}\left( \left\{\frac{\theta_{\ba}}{\theta_\bx},\frac{\theta_{\bb}}{\theta_\bx},\frac{\theta_{\bc}}{\theta_\bx}\right\}\right) = 
	\theta(\ba:\bb:\bc) - \theta(\ba:\bb:\bx) - \theta(\ba:\bx:\bc) - \theta(\bx:\bb:\bc)
	$$
	where $\Res:\tilde K^M_3(F_N) \to \tilde K^M_2(k_N)$ is the residue map in Theorem \ref{SRL} and we put 
	$$
	\theta(\ba:\bb:\bc) := \{g_{\ba-\bb},g_{\bc-\ba}\} +  \{g_{\bc-\ba},g_{\bb-\bc}\} +  \{g_{\bb-\bc},g_{\ba-\bb}\} \in \tilde K^M_2(k_N).
	$$
	
\end{lemma}

\begin{proof} 
	This follows from calculations based on the definition of the residue map and the properties of the Steinberg symbol.
	More explicitly, we use the following computation
	\be
	\partial_v (\{\pi_v^{r_1} u_1, \pi_v^{r_2}u_2, \pi_v^{r_3}u_3 \})=
	r_1 \{\overline u_2, \overline u_3\}-r_2\{\overline u_1, \overline u_3 \} + r_3 \{ \overline u_1, \overline u_2\}
	\ee
	where $v$ is a discrete valuation on $F_N$ which is trivial on $k_N$ and $\pi_v$ is a uniformizer with $v$-adic units $u_1, u_2, u_3$ (with notations from Lemma \ref{boundary}).
    Note that discrete valuations on $F_N$ which are trivial on $k_N$ are in bijection with prime divisors of $E(N)$,
	 which is the elliptic curve (in particular, the algebraic normal irreducible variety) over $k_N$ in the introduction:
for an irreducible normal variety $X$ and a prime divisor $D \subset X$, $\cO_{X ,D}$ is the subring of the function field $\C(X)$ defined by
\be
\cO_{X ,D}:= \{\varphi \in \C(X ) \ | \ \varphi \text{ is defined on } U \subset X \text{ open with } U \cap D\neq \phi \}	
\ee	
which is a discrete valuation of $\C(X)$ corresponding to $D$.
	The first four identities follow from Proposition \ref{dcomputation} and Corollary \ref{TS}. For the readers' convenience, we include the proof of the first identity:
	{\small{
	\be
	  {\partial}_\ba \left( \left\{\frac{\theta_{\ba}}{\theta_\bx},\frac{\theta_{\bb}}{\theta_\bx},\frac{\theta_{\bc}}{\theta_\bx}\right\}\right) 
	&=& 12N^2 \left\{ \overline{\frac{\theta_{\bb}}{\theta_{\bx}}},  \overline{\frac{\theta_{\bc}}{\theta_{\bx}}} \right\}   \\
	&=& 12N^2 \left\{ \frac{_{N}\!\Theta(a_1\tau + a_2-b_1\tau-b_2,\tau)^{12}}{_{N}\!\Theta(a_1\tau+a_2-x_1\tau-x_2,\tau)^{12}} ,  \frac{_{N}\!\Theta( a_1\tau + a_2-c_1\tau-c_2,\tau)^{12}}{_{N}\!\Theta(a_1\tau+a_2-x_1\tau-x_2,\tau)^{12}} \right\}     \\
		&=&(12N^2)^3 \displaystyle \left\{ \frac{g_{\ba-\bb}}{g_{\ba-\bx}}, \frac{g_{\bc-\ba}}{g_{\bx-\ba}} \right\}.
	\ee
	}}
	Then we compute $\Res$ using these first four identities:
	\be
	&&\Res\left( \left\{\frac{\theta_{\ba}}{\theta_\bx},\frac{\theta_{\bb}}{\theta_\bx},\frac{\theta_{\bc}}{\theta_\bx}\right\}\right)\\
	&:=&\sum_{v \in \Sigma(F_N/k_N)} \partial_v \left(\left\{\frac{\theta_{\ba}}{\theta_\bx},\frac{\theta_{\bb}}{\theta_\bx},\frac{\theta_{\bc}}{\theta_\bx}\right\} \right) \\
	&=&{\partial}_\ba \left( \left\{\frac{\theta_{\ba}}{\theta_\bx},\frac{\theta_{\bb}}{\theta_\bx},\frac{\theta_{\bc}}{\theta_\bx}\right\}\right)
	+{\partial}_\bb \left( \left\{\frac{\theta_{\ba}}{\theta_\bx},\frac{\theta_{\bb}}{\theta_\bx},\frac{\theta_{\bc}}{\theta_\bx}\right\}\right)
	+{\partial}_\bc \left( \left\{\frac{\theta_{\ba}}{\theta_\bx},\frac{\theta_{\bb}}{\theta_\bx},\frac{\theta_{\bc}}{\theta_\bx}\right\}\right)+
	{\partial}_\bx \left( \left\{\frac{\theta_{\ba}}{\theta_\bx},\frac{\theta_{\bb}}{\theta_\bx},\frac{\theta_{\bc}}{\theta_\bx}\right\}\right) \\
	&=&
	(12N^2)^3 \left(\displaystyle \left\{ \frac{g_{\ba-\bb}}{g_{\ba-\bx}}, \frac{g_{\bc-\ba}}{g_{\bx-\ba}} \right\} +
	\displaystyle \left\{ \frac{g_{\bb-\bc}}{g_{\bb-\bx}}, \frac{g_{\ba-\bb}}{g_{\bx-\bb}} \right\} +
	\displaystyle\left\{ \frac{g_{\bc-\ba}}{g_{\bc-\bx}}, \frac{g_{\bb-\bc}}{g_{\bx-\bc}} \right\} +
	\displaystyle \left\{ \frac{g_{\bx-\bc}} {g_{\bx-\ba}}, \frac{g_{\bx-\bb}}{g_{\bx-\ba}}\right\}  \right) \\
	&=&
	(12N^2)^3 \biggl(\theta(\ba:\bb:\bc) - \theta(\ba:\bb:\bx) - \theta(\ba:\bx:\bc) - \theta(\bx:\bb:\bc)\biggr).
	\ee
\end{proof}

Next we state a lemma closely related to, but with slightly different conditions from, \cite[Lemma 2.11]{Gon} which we will apply in Proposition \ref{klem} to $\theta(\ba:\bb:\bx)$ and other terms.\footnote{The first condition in \cite[Lemma 2.11]{Gon} does not apply to our version of $\theta$ and does not appear to apply to the version of $\theta$ defined in \cite{Gon} either}

Let $A$ and $B$ be abelian groups. Let \begin{align*}
\psi: \wedge^3\Z[A]\to B, \>\>\> {\ba}\wedge{\bb}\wedge{\bc}\mapsto \psi(\ba:\bb:\bc)
\end{align*} 
be a group homomorphism.

\begin{lemma}\label{gonchvariantlemma}
Let $\ba,\bb,$ and $\bc$ be elements of $A$ (and $-\ba, -\bb,-\bc \in A$ are their inverses in $A$).
Assume that the group $A$ is finite and that the map $\psi$ satisfies
\begin{itemize}
	\item $\psi(-\ba:-\bb:-\bc)=\psi(\ba:\bb:\bc)$
	\item $\psi(\ba + \bx, \bb + \bx, \bc + \bx) = \psi(\ba , \bb , \bc )$ for any $\bx\in A$.
\end{itemize}
Then modulo 2-torsion one has \be
\sum_{\bx \in A} \psi(\ba:\bb:\bx)=\sum_{\bx \in A}\psi(\ba:\bx:\bc)=\sum_{\bx \in A}\psi(\bx:\bb:\bc) =0.
\ee
\end{lemma}
\begin{proof}
The proof (if not the statement) is identical to the proof of \cite[Lemma 2.11]{Gon}.  Indeed, by the group homomorphism structure of $\psi$, it suffices to prove only 
$\sum_{\bx \in A} \psi(\ba:\bb:\bx)=0$.  But 
\be
\sum_{\bx \in A} \psi(\ba:\bb:\bx)
&=&\sum_{\bx \in A}\psi(\ba-(\ba+\bb):\bb-(\ba+\bb):\bx-(\ba+\bb))\\
&=&\sum_{\by \in A}\psi(-\bb:-\ba:-\by)=-\sum_{\by \in A}\psi(-\ba:-\bb:-\by)=-\sum_{\by \in A}\psi(\ba:\bb:\by).
\ee
\end{proof}
We apply the lemma when $A=A(N):= \frac{1}{N}\Z^2/\Z^2$.
\begin{proposition}\label{klem}
	Let $N\ge 3$ an integer, and let $\ba,\bb,\bc \in A(N)-(0,0)$ be distinct non-zero elements. We have
	\be
	\{g_{\ba-\bb},g_{\bc-\ba}\} +  \{g_{\bc-\ba},g_{\bb-\bc}\} +  \{g_{\bb-\bc},g_{\ba-\bb}\} =0 \quad \text{in} \quad \tilde K^M_2(k_N).
	\ee
\end{proposition}

\begin{proof}  
	The proof follows from Lemma \ref{bcomputation}, and the following observation:
	Summing the terms in Lemma \ref{bcomputation} over $\bx \in A(N)$ and applying Lemma \ref{gonchvariantlemma} 
	to $\theta(\ba:\bb:\bx), \theta(\ba:\bx:\bc),$ and $\theta(\bx:\bb:\bc)$, we obtain
	$$
	\frac{1}{N^2(12N^2)^3} \sum_{\bx \in A(N)} {\Res}\left( \left\{\frac{\theta_{\ba}}{\theta_\bx},\frac{\theta_{\bb}}{\theta_\bx},\frac{\theta_{\bc}}{\theta_\bx}\right\}\right)=\{g_{\ba-\bb},g_{\bc-\ba}\} +  \{g_{\bc-\ba},g_{\bb-\bc}\} +  \{g_{\bb-\bc},g_{\ba-\bb}\}.
	$$
	Then, by Suslin's reciprocity law it follows that $\{g_{\ba-\bb},g_{\bc-\ba}\} +  \{g_{\bc-\ba},g_{\bb-\bc}\} +  \{g_{\bb-\bc},g_{\ba-\bb}\}= 0$.

\end{proof}

\begin{theorem}\label{sres}
	Let $\ba, \bb, \bc \in \Q^2/\Z^2-(0,0)$ and suppose $\ba + \bb + \bc = 0$.   Then
	$$
	\{g_\ba,g_\bb\} + \{g_\bb,g_\bc\} + \{g_\bc,g_\ba\} = 0\ \ \hbox{\rm in $\tilde K^M_2(k_\infty) $}.
	$$
\end{theorem}
	\begin{remark}\label{srescomment}
	
	Note that Theorem \ref{sres} is still true for $\ba, \bb, \bc$  
	associated to $N=1,2$: the $N=1$ case holds trivially (as $g_{(0,0)}=1$) and the $N=2$ case has only one Manin relation, which follows directly from the $N=4$ case by observing that
	$g_{(0,1/2)}, g_{(1/2,0)}, g_{(1/2,1/2)}$
	have the same $q$-expansions as $g_{(0,2/4)}, g_{(2/4,0)}, g_{(2/4,2/4)}$ respectively.  
\end{remark}

\begin{remark} 
We remark that Theorem \ref{sres} is equivalent to 
\cite[Theorem 4.1]{Bru2} modulo torsion, which strengthens\footnote{It removes the assumption that $3$ does not divide $N$ in \cite{Bru}.} \cite[(7) Theorem 1.4, (5) Definition 1.2]{Bru}, the main result of \cite{Bru}.
Indeed:
\begin{itemize}
	\item Let $Y(N)^{(1)}$ denote the set of points of $Y(N)$ of codimension $1$. It follows e.g. from \cite[Section 2.1, especially (2.5)]{SV} that the kernel of the tame symbol map
	$K^M_2(k_N) \xrightarrow{\partial} \bigoplus_{D\in Y(N)^{(1)}} K_1^M(\C(D))$ is isomorphic to the motivic cohomology group $H^2(Y(N), 2)$.
	\item We identify motivic cohomology groups with the higher Chow groups 
	\be
	\mathrm{CH}^p(Y(N), n) = H^{2p-n} (Y(N), p)
        \ee
         as per standard practice in the literature, and we observe (e.g. \cite[p. 177]{Totaro} or \cite{Bloch}) that the generalized Grothendieck Riemann Roch theorem 
	\be
	K_n(Y(N)) \otimes \Q \simeq \bigoplus_p \mathrm{CH}^p(Y(N), n) \otimes \Q
	\ee
	induces an embedding 
		$H^2(Y(N),2)\otimes \Q \hookrightarrow K_2(Y(N))\otimes \Q$.
\end{itemize}
Since the $g_\ba$ are units (and hence have trivial divisor) for $\ba$ as above,  the tame symbol of $\{g_\ba, g_\bb\}$ is zero by definition, and so $\{g_\ba, g_\bb\}$ belongs to the Quillen $K_2$ of the modular curve of the appropriate level  (modulo tensoring with $\Q$).
In this paper we use the Quillen $K$-groups of the irreducible modular curve $Y(N)$ over $\C$, whereas  Brunault \cite{Bru} uses that of the irreducible modular curve $\cM(N)$ over $\Q$ (see \eqref{cm} and \cite[p. 2]{Bru}). There is a natural embedding from the function field $\Q(\cM(N))$ to $k_N=\C(Y(N))$, which induces a map $K_2^M(\Q(\cM(N))) \to K_2^M(k_N)$. According to \cite[Lemma 6.1.3, Chap 3]{Weibel}, for every field extension $F \subset E$, the kernel of $K_2^M(F) \to K_2^M(E)$
is a torsion subgroup; this implies that $K_2^M(\Q(\cM(N)))\otimes \Q \to K_2^M(k_N)\otimes \Q$ is injective. Through the composite map $$H^2(\cM(N),2)\otimes \Q \hookrightarrow K_2^M(\Q(\cM(N)))\otimes \Q \hookrightarrow K_2^M(k_N)\otimes \Q,$$ it follows that our Theorem \ref{sres} is equivalent to \cite[Theorem 4.1]{Bru2} modulo torsion.\footnote{Brunault's result is stated in $K_2^M(\Q(\cM(N)))\otimes \Z[\frac{1}{6N}]$.}

\end{remark}


\subsection{The construction and the proof}



We provide a proof for Theorem \ref{modularsymbol}.
We start from the following important result on the Beilinson-Kato distribution ${{\bm\mu}^2}$.
\begin{proposition}\label{wdefined}
The Beilinson-Kato distribution ${{\bm\mu}^2} \in \Dist(M_{2\times2}(\Q),\tilde K^M_2(k_\infty))$ has the following properties.
\begin{itemize}
\item[(a)] (Invariance under $T(\Q)$)  For any $t$ in the standard torus $T(\Q)$ of diagonal matrices of $\GL_2(\Q)$, we have $t^*({{\bm\mu}^2}) = {{\bm\mu}^2}$;

\item[(b)] (The Manin Relations)   For $S = \begin{pmatrix}0&-1\\ 1&0\end{pmatrix}$ and $R = \begin{pmatrix}0&-1\\ 1&-1\end{pmatrix}$ we have
$$
(1 + (S^{-1})^*){{\bm\mu}^2} = (1+(R^{-1})^* +(R^{-2})^*){{\bm\mu}^2} = 0.
$$
\end{itemize}

\end{proposition}

\begin{proof}  Invariance under $T(\Q)$ is an immediate consequence of the fact that the Siegel distribution is invariant under the action of $\Q^\times$.   The first Manin relation follows from the skew symmetry of the Steinberg symbol.  The second is equivalent to Theorem  \ref{sres}. 
\end{proof}

For any pair of distinct rational cusps $r,s\in \PP^1(\Q)$ of the upper half-plane we choose $\alpha\in \GL_2(\Q)$ so that $\alpha(\infty) = r$ and $\alpha(0) = s$ and define ${\bm\xi}(r\to s) \in \Dist(M_{2\times2}(\Q), \tilde K^M_2(k_\infty))$ by 
\bea \label{BKD}
{\bm\xi}(r\to s)={\bm\xi}( \alpha(\infty)\to \alpha(0)) := (\alpha^{-1})^*{{\bm\mu}^2} =(\alpha^{-1})^* \left( {\bm\xi}(\infty \to 0)\right)
\eea
where ${{\bm\mu}^2}$ is the Beilinson-Kato distribution. We note first of all that this definition does not depend on the choice of $\alpha$.  Indeed, if $\beta \in \GL_2(\Q)$ is any other element satisfying $\beta(\infty) = r$ and $\beta(0) = s$
then $\beta =  \alpha t$ for some $t\in T(\Q)$.  So by \textit{(a)} of Proposition \ref{wdefined} we have $(\beta^{-1})^* {{\bm\mu}^2} = (t^{-1} \alpha^{-1} )^*{{\bm\mu}^2}= (\alpha^{-1})^* ({t^{-1}})^*{{\bm\mu}^2} = (\alpha^{-1})^* {{\bm\mu}^2}$.

We have to prove that there is a unique $\GL_2(\Q)$-invariant modular symbol
$$
{\bm\xi} : \Delta_0 \lra \Dist(M_{2\times2}(\Q),\tilde K^M_2(k_\infty))
$$
with the property that 
$${\bm\xi}\biggl((0)-(\infty)\biggr) = {{\bm\mu}^2}. $$

\begin{proof} [The proof of Theorem \ref{modularsymbol}] The uniqueness is clear and the $\GL_2(\Q)$-invariance of ${\bm\xi}$ is an immediate consequence of the fact (due to \eqref{BKD}) that for any $\gamma\in \GL_2(\Q)$ and any pair of distinct cusps $r,s \in  \PP^1(\Q)$ we have 
$$
{\bm\xi}(\gamma r\to \gamma s) = (\gamma^{-1})^*\left( {\bm\xi}(r\to s) \right). 
$$
Consider the following exact sequence of $\Z$-modules
\be
0\to M\to \Z[\GL_2(\Q)] &\to& \triangle_0 \to 0\\
\gamma &\mapsto& \gamma\cdot (\infty \to 0),
\ee
where $M$ is the module of Manin relations; $\triangle_0$ is generated by $(\infty\to 0)$ as a $\Z[\GL_2(\Q)]$-module and 
the annihilator of $(\infty \to 0)$ is the left ideal generated by the set
$$
\left\{ \d -1 : \d\in T(\Q) \right\} \cup \left\{ 1+S, 1+ R +R^2 \right\}.
$$
%
So we only need to show that for any triple $r,s,t\in \PP^1(\Q)$ of distinct cusps we have 
$$
{\bm\xi}(r\to s) + {\bm\xi}(s\to r) = {\bm\xi}(r\to s) + {\bm\xi}(s\to t) + {\bm\xi}(t\to r) = 0.
$$
But there is a unique $\alpha \in \GL_2(\Q)$ that sends the ordered triple $(\infty, 0, 1)$ to $(r,s,t)$ and with that choice of $\alpha$ we have, by the Manin relations (\textit{(b)} of Proposition \ref{wdefined}),
$$
{\bm\xi}(r\to s) + {\bm\xi}(s\to r) = (\alpha^{-1})^*\bigl({\bm\xi}(\infty \to 0) + {\bm\xi}(0 \to \infty)\bigr) = (\alpha^{-1})^*(1+(S^{-1})^*){{\bm\mu}^2} = 0,
$$
and also
\be
{\bm\xi}(r\to s) + {\bm\xi}(s\to t) + {\bm\xi}(t\to r) 
&=& (\alpha^{-1})^*\bigl({\bm\xi}(\infty \to 0) + {\bm\xi}(0 \to 1) + {\bm\xi}(1 \to \infty)\bigr)\\
&=& (\alpha^{-1})^*(1+(R^{-1})^*+(R^{-2})^*){{\bm\mu}^2} = 0.
\ee
The theorem follows at once.
\end{proof}

\section{A generalization to $\GL_n(\Q)$-invariant modular symbols}\label{glngeneral}

\subsection{The statement for $\GL_n(\Q)$-invariant modular symbols}


Let $\operatorname{Vect}$ be the category whose objects are finite dimensional vector spaces over $\Q$ and whose morphisms are $\Q$-linear isomorphisms.
Let $\operatorname{Ab}$ be the category of abelian groups.
Given any linear isomorphism $\varphi:V_1 \xrightarrow{\simeq} V_2$ of two vector spaces $V_1, V_2$ over $\Q$, we define
\bea \label{test}
\varphi^*: \cS(V_2) \xrightarrow{\simeq} \cS(V_1), \quad \varphi^* (f) (x) := f(\varphi(x)).
\eea
This implies that $\cS: \operatorname{Vect} \to \operatorname{Ab}$ is a contravariant functor.\footnote{Note that if one uses $\Q$-linear maps instead of $\Q$-linear isomorphisms in the category $\operatorname{Vect}$, then $\cS$ does not form a functor.} This functor has a nice additive property: if $V\simeq V_1\oplus V_2$, then $\cS(V)\simeq \cS(V_1) \otimes_\Z \cS(V_2)$.
If $M$ is a covariant functor from $\operatorname{Vect}$ to $\operatorname{Ab}$, then any linear isomorphism $\varphi:V_1\xrightarrow{\simeq} V_2$ induces an isomorphism of abelian groups:
\bea \label{indcov}
\varphi_*: \operatorname{\Dist}(V_1, M) \to \operatorname{\Dist}(V_2, M), \quad \varphi_*(\mu) (f_2)
:=\left( \mu(\varphi^*(f_2)) \right),
\eea
where $f_2 \in \cS(V_2)$, $\mu\in \operatorname{\Dist}(V_1, M)$.
Therefore, the assignments $V \mapsto \operatorname{\Dist}(V, M)$ and $\varphi \mapsto \varphi_*$ give us a covariant functor.

Let $V$ be an $n$-dimensional vector space and $V^*=\Hom(V,\Q)$ the dual vector space. Denote $\Aut_\Q(V)$ by $G$. Let $\cB$ be the set of ordered bases of $V^*$.
Now we give a definition of $G$-invariant modular symbols following \cite{AR}.
Let $G$ act on $V$ on the left. We define a right action of $G$ on $V^*$ as follows:
\be
(\g^* \l)(u ) = \l (\g u), \quad \g \in G, \l \in V^*, u \in V,
\ee
and $G$ acts on $(V^*)^n$ diagonally.
Recall the distribution ${\bm\mu}^n$ in Definition \ref{glnSD}.
\begin{definition}\label{glnm}
For any $\ulambda = (\lambda_1,\lambda_2,\ldots,\lambda_n) \in (V^\ast)^n$, we define a distribution (using \eqref{indcov})
\bea \label{keydef}
\xi_\ulambda \in \Dist(V \times  V, \tilde K^M_n(k_\infty)), \quad 
\xi_\ulambda := \begin{cases}
\phi(\ulambda)^{-1}_\ast ({\bm\mu}^n)  & \hbox{ \rm if\  }\ulambda\in \cB\\
\\
			 0 &\hbox{\rm otherwise,}
\end{cases}
\eea
where $\phi(\ulambda): V\times V\xrightarrow{\simeq} M_{n\times 2}(\Q)$ is the $\Q$-linear isomorphism induced from $\ulambda$:
\be
 V\times V \ni (v_1, v_2) \mapsto 
\begin{pmatrix}
\l_1(v_1) & \l_1(v_2) \\
\l_2(v_1) & \l_2(v_2) \\
\vdots & \vdots \\
\l_n(v_1) & \l_n(v_2) 
\end{pmatrix} \in M_{n\times 2}(\Q).
\ee
In particular, we have 
$$
\xi_\ulambda ([a+ rL_{\ulambda}]\otimes [b+rL_{\ulambda}]) := \begin{cases}
\{g_{\lambda_1(a)/r, \lambda_1(b)/r}, \ldots,g_{\lambda_n(a)/r, \lambda_n(b)/r} \}  & \hbox{ \rm if\  }\ulambda\in \cB\\
\\
			 0 &\hbox{\rm otherwise,}
\end{cases}
$$
where $a, b \in V$, $r \in \N$, and $L_\ulambda =\{ v \in V : \ulambda (v)=(\lambda_1(v), \cdots, \lambda_n(v)) \in \Z^n\}$. 
\end{definition}

\begin{definition}\cite{AR}\label{AR}
For any (contravariant) $G$-module $M$, the map 
\be
\xi:  (V^*)^n \to M
\ee
 is called a $G$-invariant modular symbol with values in $M$, if $\xi$ satisfies the following properties:
\begin{itemize}
\item $\xi$ is $G$-equivariant, i.e. for all $\gamma\in G$ and $\ulambda\in \cB$ we have 
\bea \label{one}
\xi({\gamma^\ast \ulambda}) = \gamma^\ast \left( \xi(\ulambda)\right).
\eea
\item For all $\ulambda \in \left(V^*\right)^n$ and $(t_1,\ldots, t_n)\in \left(\Q^\times\right)^n$ we have
\bea \label{two}
\xi(t_1\lambda_1\ldots, t_n\lambda_n) = \xi(\lambda_1,\ldots,\lambda_n).
\eea
\item For any permutation $\sigma \in S_n$ and any $\ulambda = (\lambda_1,\ldots,\lambda_n) \in \left(V^*\right)^n$ we have
\bea \label{three}
\xi(\lambda_{\sigma(1)},\lambda_{\sigma(2)},\ldots, \lambda_{\sigma(n)}) = \sign(\sigma)\cdot 
\xi(\lambda_1, \lambda_2,\ldots,\lambda_n).
\eea
\item If $\ulambda\notin \cB$ (i.e. the $\lambda_1,\ldots,\lambda_n$ are linearly dependent in $V^*$), 
then 
\bea \label{five}
\xi(\ulambda)=0.
\eea
\item  For any $(n+1)$-tuple $(\lambda_0,\lambda_1,\ldots,\lambda_n) \in \left(V^*\right)^{n+1}$ we have
\bea \label{four}
\sum_{i=0}^n (-1)^i\xi(\lambda_0,\ldots,\hat \lambda_i,\ldots,\lambda_n) = 0.
\eea
\end{itemize}
\end{definition}

Recall that $G=\Aut_\Q(V)$. We equip $\tilde K^M_n(k_\infty)$ with a trivial contravariant $ G$-module structure, i.e. $G$ acts on $\Dist(V \times  V, \tilde K^M_n(k_\infty))$ as follows (see \eqref{distributionaction}):
\bea \label{taction}
\int \gamma^\ast f \cdot d\gamma^\ast\mu = \int f(x)d\mu(x), \quad (\g^\ast f)(v_1, v_2) = f(\g v_1, \g v_2), \quad x, (v_1,v_2) \in V \times V,
\eea
for any $\mu\in \Dist(V\times V,\tilde K^M_n(k_\infty))$, $f\in \cS(V\times V)$ and $ \gamma\in  G$.

\begin{theorem} \label{mainthm}
Let $\xi_V  : (V^\ast)^n  \lra  \Dist(V \times  V, \tilde K^M_n(k_\infty))$ be defined by
$$
\begin{array}{ccccc}
\xi_V  &: &(V^\ast)^n  &\lra&  \Dist(V \times  V, \tilde K^M_n(k_\infty))\\
&&\ulambda &\longmapsto& \xi_V(\ulambda)=\xi_\ulambda.\\
\end{array} 
$$
Then $\xi$ is a $G$-invariant modular symbol with values in $\Dist(V \times  V, \tilde K^M_n(k_\infty))$.
\end{theorem}
The above theorem reduces to Theorem \ref{glnmodularsymbol} once one fixes an isomorphism between $(V^*, \ulambda)$ and $((\Q^n)^*, \{e_1^*, \cdots, e_n^*\})$:
\be
\phi(\ulambda)_\ast \left( \xi_V(\ulambda)\right)  = \bm\xi_n (e_1^*, \cdots, e_n^*)=\bm\mu^n, \quad \ulambda \in \cB,
\ee
using the $\Q$-linear isomorphism $\phi(\ulambda): V\times V\xrightarrow{\simeq} M_{n\times 2}(\Q)$ induced from $\ulambda$.

\begin{proposition}\label{cgcom}
The statement for $n=2$ of Theorem \ref{glnmodularsymbol} is equivalent to Theorem \ref{modularsymbol}.
\end{proposition}
\begin{proof}
Define the slope of a line (one-dimensional subspace) $\ell$ of $\Q^2$ to be
\be
\mathrm{slope}(\ell):=\left\{\frac{x}{y}\in \Q \sqcup \{\infty\} : (x,y) \in \ell \subset \Q^2 \right\}.
\ee
We identify
\be
\BP^1(\Q):=\{ \ell \subset \Q^2 :  \text{ $\ell$ is a line of $\Q^2$} \}
 &\xrightarrow{\sim}&   \Q \sqcup \{\infty\}\\
\ell &\mapsto& \mathrm{slope}(\ell).
\ee
Note that $\GL_2(\Q)$ acts on $\Q \sqcup \{\infty\}$:
\be
\begin{pmatrix}
a & b\\
c& d
\end{pmatrix} 
\cdot z = \frac{az+b}{cz+d}, \quad z \in \Q \sqcup \{\infty\}, \begin{pmatrix}
a & b\\
c& d
\end{pmatrix}  \in \GL_2(\Q).
\ee
Let $\cB_2$ be the set of ordered basis of $\Hom(\Q^2,\Q)$.
Then $(\Q^\times)^2$ acts on $\cB_2$:
\be
(x_1,x_2)^*(\lambda_1,\lambda_2)= (x_1^*\lambda_1, x_2^*\lambda_2),
\quad (x^* \lambda)(v)=\lambda(x\cdot v)
\ee
where $x_1, x_2,x \in \Q^\times$ and $(\lambda_1,\lambda_2) \in \cB_2$.
Here $[(\lambda_1,\lambda_2)]$ denotes the equivalence class of $(\lambda_1,\lambda_2)$ under the $(\Q^\times)^2$-action.
Then the assignment
\bea \label{assingn}
(\lambda_1,\lambda_2) \mapsto (\mathrm{slope}(\ker(\lambda_1))-\mathrm{slope}(\ker(\lambda_2)))
\eea
induces a map 
$$
\nu: \cB_2/(\Q^\times)^2 \to \Delta_0=\mathrm{Div}^0(\Q\sqcup \{\infty\})) 
$$
such that  $\nu (\g^* [(\lambda_1, \lambda_2)]) = \g^{-1}\cdot \nu([(\lambda_1,\lambda_2)])$ for $\g \in \GL_2(\Q)$. Note that if $(\lambda_1, \lambda_2)$ is not in $\cB_2$, $(\lambda_1, \lambda_2)$ goes to zero under $\nu$.
Unravelling the definitions, we have
\bea \label{rln}
 \bm\xi_2([\lambda_1, \lambda_2]) = (\phi(\ulambda)_\ast)^{-1} \bm\xi(\nu([\lambda_1, \lambda_2])).
\eea
The Manin relations for $\bm\xi$ correspond to \eqref{three} and \eqref{four} of $\bm\xi_2$. Thus we can construct $\bm\xi_2$ from $\bm\xi$.
Conversely, we can reverse the process to construct $\bm\xi$ from $\bm\xi_2$ by using the fact that $\Delta_0$ is generated by $(\infty \to 0)$ as a left $\Z[\GL_2(\Q)]$-module modulo the Manin relations.
Thus it follows from the uniqueness of both constructions that $\bm\xi_2$ is equivalent to $\bm\xi$.
\end{proof}

\subsection{$G$-equivariance}

We will prove \eqref{one}.
Since $\phi(\gamma^\ast \ulambda) = \phi(\ulambda \circ \gamma)$,
we have 
\bea \label{middle}
\phi(\gamma^\ast \ulambda)^{-1} = \gamma^{-1} \circ \phi(\ulambda)^{-1}: M_{n\times 2}(\Q) \to V \times V.
\eea
If $\ulambda \notin\cB$, the $G$-equivariance is obvious.
For $\ulambda \in \cB$, the following computation gives \eqref{one}:
\be 
\xi_V({\gamma^\ast \ulambda}) (f)  &=& \left( (\phi(\gamma^\ast \ulambda)^{-1} )_\ast ({\bm\mu}^n) \right) (f) \quad \text{by} \quad \eqref{keydef}\\
&=& {\bm\mu}^n (f \circ \phi(\gamma^\ast \ulambda)^{-1})  \quad \text{by} \quad \eqref{test}, \eqref{indcov}\\
&=& {\bm\mu}^n ( \left(f \circ \gamma^{-1}\right) \circ \phi(\ulambda)^{-1})  \quad \text{by} \quad \eqref{middle}  \\
&=& (\phi(\ulambda)^{-1}_\ast({\bm\mu}^n)) (f \circ \gamma^{-1}) \quad \text{by} \quad  \eqref{test}, \eqref{indcov},\eqref{keydef} \\
 &=&\left( \gamma^\ast \left( \xi_V(\ulambda)\right)\right)(f) \quad \text{by} \quad \eqref{taction}.
\ee

\subsection{Other basic properties of $\xi_V$}
For the proof of \eqref{two}, observe that \eqref{keydef} implies, for $a,b \in V$, $t_1, \ldots, t_n \in \Q^\times$ and $\ulambda \in \cB$, 
\be
\xi_\ulambda ([a+rL_{t_1\lambda_1, \ldots, t_n \lambda_n}]\times [b+rL_{t_1\lambda_1, \ldots, t_n \lambda_n}])=\{g_{\lambda_1(at_1/r), \lambda_1(bt_1/r)}, \ldots,g_{\lambda_n(at_n/r), \lambda_n(bt_n/r)}  \} 
\ee
where $
L_{t_1\lambda_1, \ldots, t_n \lambda_n}=\{x \in V : t_i\lambda_i(x) \in \Z, \quad i=1, \ldots,n \}.$
Since \eqref{keydef} also implies that
\be
\xi_{t_1\lambda_1, \ldots, t_n \lambda_n} ([a+rL_{t_1\lambda_1, \ldots, t_n \lambda_n}]\times [b+rL_{t_1\lambda_1, \ldots, t_n \lambda_n}])=\{g_{\lambda_1(at_1/r), \lambda_1(bt_1/r)}, \ldots,g_{\lambda_n(at_n/r), \lambda_n(bt_n/r)}  \},
\ee
\eqref{two} is proved.

\noindent Equation \eqref{three} follows from the antisymmetry of the Milnor $K_n$-group, and  \eqref{five} follows from the  definition of $\xi_V$.

\subsection{The cocycle relation}

{

\begin{proposition}\label{cocyclep}
Let $\ba_j \in A(N)-(0,0)$ for $j=0, \ldots, n$.
If $\ba_0=\ba_1 + \cdots + \ba_n$, then
\bea \label{cocycle}
\sum_{i=0}^n(-1)^i \{g_{\ba_0}, \ldots, \widehat{g_{\ba_i}}, \ldots, g_{\ba_n}  \} =0
\eea
in $\tilde K^M_n(k_\infty)$. 
\end{proposition}

\begin{proof}
The proof goes by induction on $n$. The case $n=1$ is obvious.
The case $n=2$ follows from the Manin relation (Proposition \ref{klem}): if $-\ba_0 + \ba_1 +\ba_2 =0$, then one can easily check that \eqref{cocycle} is equivalent to
Theorem \ref{sres} by using the fact that $\{g_{-\ba_0},g_{\ba_i}\} \equiv \{g_{\ba_0}, g_{\ba_i}\}\equiv - \{g_{\ba_i}, g_{\ba_0} \}$ for $i=1,2$.

Now suppose $m\geq 2$ and the result is true for $n = m$. 
We will prove the result for $n=m+1$. 
Set $\ba_0=\ba_1+\cdots+\ba_{m+1}$, which we rewrite as
 $\ba_0'=\ba_1'+\cdots+ \ba_m'$ where
\be
\ba_r':=
\begin{cases}
\ba_0 -\ba_{m+1}, & \text{ if } r=0 \\
\ba_r, &  \text{ if }  1 \leq r \leq m.
\end{cases}
\ee

Then by the induction hypothesis for $n=m$ we have
\begin{eqnarray}\label{ih}
\sum_{i=0}^m (-1)^i \left\{g_{\ba_0'}, \ldots, \widehat{g_{\ba_i'}}, \ldots, g_{\ba_m'} \right\} =0.
\end{eqnarray}
Multiplying \eqref{ih} by $\{g_{\ba_{m+1}}\}$ and $\{g_{\ba_0}\}$, we get
\begin{eqnarray*}
A&:=&\sum_{i=0}^m (-1)^i \left\{g_{\ba_{m+1}}, g_{\ba_0'}, \ldots, \widehat{g_{\ba_i'}}, \ldots, g_{\ba_m'}  \right\} = 0,\\
B&:=&\sum_{i=0}^m (-1)^i \left\{g_{\ba_0}, g_{\ba_0'}, \ldots, \widehat{g_{\ba_i'}}, \ldots, g_{\ba_m'}  \right\} = 0.
\end{eqnarray*}
Using the fact that  $\ba_0=\ba_0'+\ba_{m+1}$ we obtain
\bea\label{stsymbol}
\{g_{\ba_0'}, g_{\ba_{m+1}}\} - \{g_{\ba_0},g_{\ba_{m+1} }\} + \{g_{\ba_0},g_{ \ba_0'}\} =0.
\eea
We now compute $A-B$:
\be
A-B &=&
\left\{g_{\ba_{m+1}}, g_{\ba_1}, \ldots, g_{\ba_m} \right\}  - \left\{g_{\ba_0}, g_{\ba_1}, \ldots, g_{\ba_m} \right\}\\
&&+ \sum_{i=1}^m (-1)^i \left( \{ g_{\ba_{m+1}}, g_{\ba_0'}\} - \{g_{\ba_0},g_{ \ba_0'}\}\right) \cdot \left\{g_{\ba_1}, \ldots, \widehat{g_{\ba_i}}, \ldots, g_{\ba_{m}}  \right\}
\\
&\stackrel{\eqref{stsymbol}}{=}&\left\{g_{\ba_{m+1}}, g_{\ba_1}, \ldots, g_{\ba_m} \right\} - \left\{g_{\ba_0}, g_{\ba_1}, \ldots, g_{\ba_m} \right\} +\sum_{i=1}^m (-1)^i \left\{g_{\ba_{m+1}}, g_{\ba_0}, \ldots, \widehat{g_{\ba_i}}, \ldots, g_{\ba_{m}}  \right\}\\
&=&(-1)^m  \sum_{i=0}^{m+1} (-1)^i \left\{g_{\ba_0}, \ldots, \widehat{g_{\ba_i}}, \ldots, g_{\ba_{m+1}} \right\}
\ee
Since $A-B=0$, we get the desired assertion for $n=m+1$.
Thus the proposition follows by induction.
\end{proof}

We need to prove \eqref{four} of Theorem \ref{mainthm}. 
\begin{definition}
A set $\{v_1, \ldots, v_r\}$ of vectors in a $\Q$-vector space is said to be in general position if
the condition  
\be
c_1 v_1 + \cdots +c_r v_r = 0, \quad c_1, \ldots, c_r \in \Q
\ee
implies that either $c_1=\cdots =c_r =0$ or $c_1 \cdots c_r \neq 0$.
\end{definition}
This is equivalent to saying that any proper subset of $\{v_1, \ldots, v_r\}$ is linearly independent.
\begin{lemma}\label{flem}
If $\{\l_0, \ldots, \l_n\} \subset V^*$ is in general position and $\l_0=\l_1 +\cdots + \l_n$, then the cocycle relation \eqref{four} holds.
\end{lemma}
\begin{proof}
For any subset $I:=\{i_1, \ldots, i_n\}\subset \{0, 1, \ldots, n\}$ of cardinality $n$, let $\l_I:=\{\l_{i_1}, \ldots, \l_{i_n}\} \subset V^*$.
Then $\l_I$ is linearly independent.
For any $u, v \in V$ and $r \in \Q^\times$, we have (see Definition \ref{glnm})
\bea \label{computation}
\xi_{\l_I} \left([u+rL_{\l_I}] \otimes [v+rL_{\l_I}]\right)
=\{g_{\ba_{i_1}}, \ldots, g_{\ba_{i_n}} \}, \quad
\ba_{i}:=(\l_{i}(u)/r,\l_{i}(v)/r) \in A(N), \quad i \in I,
\eea
where we recall $L_{\l_I}=\{ x \in V : \l_{i}(x) \in \Z, \forall i \in I\}$ and $[u+r L_{\l_I}] \otimes [v+r L_{\l_I}] \in \cS(V)\otimes \cS(V)$ is the tensor product of the characteristic functions on the affine lattices $u+r L_{\l_I}$ and $v+r L_{\l_I}$.
Note that $\{ [u+rL_{\l_I}] \otimes [v+rL_{\l_I}] : u, v\in V, r \in \Q^\times\}$ generates $\cS(V)\otimes \cS(V)$ as an abelian group and $\l_0=\l_1 +\cdots + \l_n$ implies that $\ba_0=\ba_1+ \cdots + \ba_n$. Therefore, \eqref{four} reduces to Proposition \ref{cocyclep}.
\end{proof}


We need to prove \eqref{four} when $\{\l_0, \ldots, \l_n\}$ is not in general position. To this end, we need the following lemma.
\begin{lemma}
There exists a subset $J =\{j_0, \ldots, j_k\} \subset \{0, \ldots, n\}$ such that $\l_J=\{\l_{j_0}, \ldots, \l_{j_k} \} \subset V^*$ is in general position and linearly dependent.
\end{lemma}
\begin{proof}
Let us consider the set $S$ of all subsets $\l_J$ of $\{\l_0, \ldots, \l_n\}$ which are linearly dependent. Then $S$ is a finite non-empty set: $\{\l_0, \ldots, \l_n\}$ is always linearly dependent. Thus there exists a minimal linearly dependent subset $\l_J=\{\l_{j_0}, \ldots, \l_{j_k} \} \subset \{\l_0, \ldots, \l_n\}$. 
Because of the minimality, every proper subset of $ \l_J$ is linearly independent, which implies that $\l_J$ is in general position.
\end{proof}

Using this lemma, we choose a set $\{\l_{j_0}, \ldots, \l_{j_k} \} $ which is linearly dependent and in general position. Note that this set does not have to be unique and $k$ could be zero (this is the case if $0 \in \{\l_1, \ldots, \l_n\}$, and we get $\{\l_{j_0}\}=\{0\}$). By using the homogeneity \eqref{two} and the anti-symmetry \eqref{three}, we may assume that $\{\l_{0}, \ldots, \l_{k}\}$ is in general position and linearly dependent with relation $\l_0=\l_1+ \cdots +\l_k$ (if $k=0$, then $\l_0=0$).

%
%
If $k=n$, then Lemma \ref{flem} implies \eqref{four}.
If $k < n$, then $\{ \l_0,  \ldots, \hat \l_i, \ldots, \l_{n}\}$ is linearly dependent for any $ i \geq k+1$, because $\{\l_0, \ldots, \l_k\}$ is linearly dependent. Thus \eqref{five} says that
\bea \label{firstn}
\xi_{(\l_0, \ldots, \hat \l_i, \ldots, \l_{n})}:=\xi_V(\l_0,  \ldots, \hat \l_i, \ldots, \l_{n} )=0, \quad \forall i \geq k+1.
\eea
Let $\l^{(i)}=\{\l_0,  \ldots, \hat \l_i, \ldots, \l_{n}  \}$. For $i \leq k$, then (the same notation as \eqref{computation})
\bea \label{anothercon}
\xi_{(\l_0,  \ldots, \hat \l_i, \ldots, \l_{n} )}([u+rL_{ \l^{(i)}}]\otimes [v+rL_{ \l^{(i)}} ])=
\begin{cases}
0, & \text{ if } \l^{(i)} \text { is linearly dependent,} \\
\{g_{\ba_{0}}, \ldots, \hat g_{\ba_i}, \ldots, g_{\ba_n} \} &  \text{ otherwise. }
\end{cases}
\eea
Since the vector space generated by $\{\l_0, \ldots, \hat\l_i, \ldots, \l_k\}$ is independent of $i \leq k$, the vector space generated by $\l^{(i)}$ is also independent of $i \leq k$.

Since $\{\l_0, \ldots, \l_k\}$ is in general position and $\l_0=\l_1 +\cdots + \l_k$, a similar argument as in Lemma \ref{flem} implies that
\bea \label{secondn}
\sum_{i=0}^k (-1)^i  \{g_{\ba_{0}}, \ldots, \hat g_{\ba_i}, \ldots, g_{\ba_k} \}=0.
\eea

Finally, we compute
\be
&&\sum_{i=0}^n (-1)^i\xi_{(\lambda_0,\ldots,\hat \lambda_i,\ldots,\lambda_n)} ([u+rL_{ \l^{(i)}}]\otimes [v+rL_{ \l^{(i)}} ])\\
&\stackrel{\eqref{firstn}}{=}&\sum_{i=0}^k (-1)^i \xi_{(\l_0, \ldots, \hat \l_i, \ldots, \l_{n})}([u+rL_{ \l^{(i)}}]\otimes [v+rL_{ \l^{(i)}} ]) \\
&\stackrel{\eqref{anothercon}}{=}&\sum_{i=0}^k (-1)^i  \{g_{\ba_{0}}, \ldots, \hat g_{\ba_i}, \ldots, g_{\ba_n} \} \quad (\text {by using the comment below \eqref{anothercon}} \\
&=&\left( \sum_{i=0}^k (-1)^i  \{g_{\ba_{0}}, \ldots, \hat g_{\ba_i}, \ldots, g_{\ba_k} \}\right)\cdot\{g_{\ba_{k+1}}, \ldots, g_{\ba_n} \} 
\stackrel{\eqref{secondn}}{=}0.
\ee
From this, we conclude \eqref{four}.

\section{Appendix} \label{appendix}
We provide a short explanation on the algebraicity of $\pi_N:\cE_N \to Y_N$ in \eqref{univell}.
\begin{theorem}
    Let $N \geq 3$. 
    \begin{enumerate}[(1)]
            \item There is a smooth affine curve $Y(N)$ over $\C$ such that $Y(N)(\C)$ is isomorphic to the complex manifold $Y_N$, and $Y_N$ is bijective with the set of isomorphism classes of triples $(E,e_1,e_2)$ where $E$ is an elliptic curve over $\C$ and $(e_1,e_2)$ is a pair of sections of $E$ over $\Spec(\C)$ which forms a $\Z/N\Z$-basis of $\Ker(N:E \to E)$ such that the Weil pairing of $e_1$ and $e_2$ is a fixed $N$-th root of unity.
        The smooth compactification $X(N)$ of $Y(N)$ is a smooth projective curve over $\C$ such that  $X(N)(\C)$ is bijective with $Y_N \sqcup \{\Gamma(N)\text{-cusps}\}$.
        \item There is a smooth affine curve $Y(N)$ over $\C$ such that $Y(N)(\C)$ is isomorphic to the complex manifold $Y_N$. 
        The smooth compactification $X(N)$ of $Y(N)$ is a smooth projective curve over $\C$ such that  $X(N)(\C)$ is bijective with $Y_N \sqcup \{\Gamma(N)\text{-cusps}\}$.
    \item   There is an elliptic surface $\overline \cE(N)$ over $X(N)$ (in the sense of \cite[Definition on p 202]{Silverman}) such that $\cE_N:=\tilde \fH/\tilde \Gamma(N)$ is isomorphic to the complex manifold $\cE(N)(\C)$ associated to $\cE(N):=Y(N) \times_{X(N)} \overline \cE(N)$.
    \end{enumerate}
\end{theorem}
\begin{proof}
{\it(1)}. Following \cite[Section 1.1]{Kat}, for $N\geq 3$, let $\cM(N)$ be the modular curve over $\Q$ of level $N$ without cusps, which represents the functor sending $S$ to the set of isomorphism classes of triples $(E, e_1,e_2)$ where $E$ is an elliptic curve over $S$
and $(e_1,e_2)$ is a pair of sections of $E$ over $S$ which forms a $\Z/N\Z$-basis of $\Ker (N : E\to E)$; the representability was proven in \cite{DR}; we also refer to \cite[Proposition 1.10.12, Theorem 4.7.0, Corollary 4.7.1, and Corollary 4.7.2]{KatzMazur} for details on representability.
Then $\cM(N)$ is a smooth irreducible affine curve over $\Q$. The total constant field of $\cM(N)$ (the field of all algebraic numbers in the affine ring $\cO(\cM(N))$) is not $\Q$, but is generated over $\Q$ by a primitive $N$-th root $\zeta_N$ of 1. Let $\overline\cM(N)$ be the smooth compactification of $\cM(N)$. By \cite[Section 1.8]{Kat}, the canonical map
\be
\nu:\fH \to \cM(N)(\C); \tau \mapsto (\C/\Z\tau+\Z,\tau/N\! \! \! \mod \Z\tau+\Z, 1/N\!\! \! \mod\Z\tau+\Z)
\ee
induces an isomorphism
\bea \label{cm}
(\Z/N\Z)^\times \times \fH/\Gamma(N) \xrightarrow{\simeq} \cM(N)(\C).
\eea
Note that the complex algebraic variety $\cM(N)_\C=\cM(N)\otimes_\Q \C$ is not irreducible; let $\cM(N)_\C^0$ be an irreducible component such that $\nu:\fH/\Gamma(N) \xrightarrow{\simeq} \cM(N)_\C^0(\C)$ is an isomorphism. 
Let $\overline\cM(N)_\C^{0}$ be the smooth compactification of $\cM(N)_\C^0$ corresponding to the compactification $\overline\cM(N)$ of $\cM(N)$.
We define
\begin{eqnarray*}
    Y(N):=\cM(N)_\C^0, \quad X(N):=\overline\cM(N)_\C^0.
\end{eqnarray*}
Then $X(N)$ and $Y(N)$ satisfy all the desired properties by the construction.

{\it(2)}. The existence of the elliptic surface $\overline \cE(N)$ over $X(N)$ (in the sense of \cite[Definition on p. 202]{Silverman}) was proved in \cite[Section 8, Theorem 11.5]{Kodaira} by Kodaira; this result was explicitly stated in \cite[Definition 4.1 and above]{Shioda} and \cite[Theorem 5.5]{Shioda}. 
In particular, this means that $\overline \cE(N)$ is an algebraic variety over $\C$.
Note that the smooth compactification $X(N)$ is isomorphic to $\Delta(N)$ in \cite[Theorem 5.5]{Shioda}; see \cite[(4.1) and below]{Shioda}.
Let $\overline \cE(N)$ be the elliptic surface $B(N)$ over $\Delta(N)$ in \cite[Theorem 5.5]{Shioda} and $\cE(N)$ be $B'(N)$ in the proof of \cite[Theorem 5.5]{Shioda}; we have the following cartesian diagram.
\[\begin{tikzcd}
	{\cE(N)}=B'(N) & {\overline \cE(N)=B(N)} \\
	{Y(N)} & {X(N)=\Delta(N).}
	\arrow[" ", from=1-1, to=1-2]
	\arrow["{}"', from=1-1, to=2-1]
	\arrow[from=1-2, to=2-2]
	\arrow[" ", from=2-1, to=2-2]
\end{tikzcd}\]

Then it was shown that $\cE(N)(\C) =B'(N)(\C)$ is isomorphic to $\cE_N:=\tilde \fH/\tilde \Gamma(N)$ in \cite[(5.7) and p. 39]{Shioda}.
\end{proof}


We recall $_{N}\!\Theta(u,\tau)^{12} :=  \Theta(u,\tau)^{12 N^2}/\Theta(Nu,\tau)^{12}$ from Definition \ref{mero} and the function field $F_N=k_N(E(N))$ in Definition \ref{fn}.
    By Corollary \ref{periodic}, $_{N}\!\Theta(u,\tau)^{12}$ is a meromorphic function on $\cE_N\simeq\cE(N)(\C)$.
    \begin{proposition} \label{rationality}
        For $N\geq 3$, the function $_{N}\!\Theta(u,\tau)^{12}$ extends to a meromorphic function on $\overline \cE(N)(\C)$. Consequently, $_{N}\!\Theta(u,\tau)^{12} \in F_N$ and $_{N}\!\Theta_\ba(u,\tau)^{12} \in F_N$ for $\ba \in A(N)$.
    \end{proposition}
    \begin{proof}
    Let $\overline \pi_N:\overline \cE_N \to X_N$ be the universal generalized elliptic curve over $X_N$, where $\overline \cE_N=\overline \cE(N)(\C)$ and $X_N=X(N)(\C)$. 
    Then $X_N=Y_N \sqcup \{\Gamma(N)-\text{cusps}\}$.
    For a $\Gamma(N)$-cusp $v$, let $C_v:=\overline \pi_N^{-1}(v)$ be the singular fiber, which is the standard N\'eron $N$-gon by \cite{Kodaira}; see \cite[Proposition 4.2, Theorem 5.2 (iii), and Example 5.4]{Shioda} for details. 
    Following \cite{Shioda} and \cite{Kodaira}, let $\theta_{v,i}$ $(0 \leq i \leq N-1)$ be the irreducible components of the divisor $C_v$ and take $\theta_{v,0}$ to be the unique component of $C_v$ containing the identity $o(v)$. 
    Let $C_v^\#$ denote the set of points of multiplicity one on $C_v$, which is shown to be a commutative algebraic group with identity $o(v)$ by \cite[Section 9]{Kodaira}; $\theta_{v,0}^\# =\theta_{v,0}\cap C_v^\#$ is the connected component of the identity. 
    Then $\theta_{v,0}^\#$ is a multiplicative subgroup and $C_v^\#/\theta_{v,0}^\# \simeq \Z/N\Z$.
    
    In order to show the meromorphicity of $_{N}\!\Theta(u,\tau)^{12}$ around $C_v$, we use the local chart of $\overline \cE_N$ around the singular fiber $C_v^\#$, given in the proof of \cite[Theorem 5.5]{Shioda}. We may assume that $v$ is the cusp at infinity $v_0$ without loss of generality, and we use the coordinate on $C_v^\#$ in \cite[p. 39]{Shioda}
    \begin{eqnarray*}
        q_N=e^{2 \pi i \tau/N}, \quad t= e^{2 \pi i u}, \quad (u, \tau) \in \C \times \fH;
    \end{eqnarray*}
    $C_v^\#$ is covered by $N$ open sets $W_i \quad (0 \leq i \leq N-1)$ of $\overline \cE_N$ with coordinates $((q_N, t))_i \quad (0 \leq i \leq N-1)$ (cf. \cite[p. 597-600]{Kodaira}). 
    Recall the following formula
        \begin{eqnarray*}
           _{N}\!\Theta(u,\tau)^{12}=q^{N^2-1} \frac{(t^{1/2} - t^{-1/2})^{12N^2} \prod_{n=1}^\infty \left(1 - q^nt\right)^{12N^2} \left(1 - q^nt^{-1}\right)^{12N^2}}{ (t^{N/2} - t^{-N/2})^{12} \prod_{n=1}^\infty \left(1 - q^n t^{N}\right)^{12}\left(1 - q^nt^{-N}\right)^{12}}
        \end{eqnarray*}
        for $t=e^{2 \pi i u}$ and $q=e^{2 \pi i \tau}$ and $(u,\tau) \in \C\times \fH$. 
         From this formula and the coordinate chart, it is clear that $_{N}\!\Theta(u,\tau)^{12}$ is well defined on $C_v$ where $[\tau] \to [i \infty]=v$ and so $q_N \to 0$, whereas $t\not\to 0, \infty$. Thus, $_{N}\!\Theta(u,\tau)^{12}$ extends to a meromorphic function on $\overline\cE_N=\cE_N \sqcup \bigsqcup_{v}C_v$. 
    \end{proof}


\begin{thebibliography}{PTW02} 
\bibitem{AR}
Ash, Avner; Rudolph, Lee: The modular symbol and continued fractions in higher dimensions. Invent. Math. 55 (1979), no. 3, 241-250.


\bibitem{AS}
Ash, Avner; Stevens, Glenn: Modular forms in characteristic $\ell$ and special values of their $L$-functions. Duke Math. J. 53 (1986), no. 3, 849-868.

\bibitem{BCG}
Bergeron, Nicolas; Charollois, Pierre; Garcia, Luis: Transgressions of the Euler class and Eisenstein cohomology of $\GL_N(\Z)$. Japanese Journal of Mathematics 15 (2020), pp. 311--379.

\bibitem{Bloch} 
Bloch, Spencer; The moving lemma for higher Chow groups. Journal of  Algebraic Geometry (1994), pp. 537-568.

\bibitem{Bru}
Brunault, Francois: Beilinson-Kato elements in $K_2$ of modular curves. Acta Arith. 134 (2008), no. 3, 283-298.

\bibitem{Bru2}
Brunault, Francois: On the $K_4$ group of modular curves. arXiv:2009.07614 (2022).


\bibitem{Colmez}
Colmez, Pierre: La conjecture de Birch et Swinnerton-dyer $p$-adique. 
S\'eminaire Bourbaki: volume 2002/2003, expos\'es 909-923, Ast\'erisque, no. 294 (2004), Talk no. 919, pp. 251-319.


\bibitem{DR}
Deligne, P.; Rapopport, M.: Les sch\'emas de modules des courbes elliptiques. Modular functions of one variable II. (Lect. Notes Math. vol. 349, pp. 143--316). Berlin-Heidelberg-New York: Springer 1973.


\bibitem{DS}
Diamond, Fred; Shurman, Jerry: 
A first course in modular forms, volume 228
of Graduate Texts in Mathematics. Springer-Verlag, New York, 2005.


\bibitem{FK}
Fukaya, Takako; Kato, Kazuya:
On conjectures of Sharifi
Kyoto J. Math. 64(2): 277--420 (May 2024)

\bibitem{Gon}
Goncharov, Alexander B.: Euler complexes and geometry of modular varieties. Geom. Funct. Anal. 17 (2008), no. 6, 1872-1914.

\bibitem{GL}
Goncharov, A. B.; Levin A. M.: Zagier's conjecture on $L(E,2)$. Inventiones mathematicae, Volume 132, pages 393-432, (1998).


\bibitem{Kat}
Kato, Kazuya:  p-adic Hodge theory and values of zeta functions of modular forms. Ast\'erisque 295 (2004), 117-290.
\bibitem{KatzMazur}
Katz, Nick; Mazur, Barry: Arithmetic moduli of elliptic curves, Annals of Mathematics Studies 108, Princeton university press.

\bibitem{Kodaira}
Kodaira, K.: On compact analytic surfaces II--III, Ann. of Math., 77 (1963), 563--626; 78 (1963), 1--40.

\bibitem{Kriz}
Kriz, Sophia: On Weil reciprocity in motivic cohomology. Available at  krizsophie.github.io/Motives22102.pdf. 

\bibitem{Lang}
Lang, Serge: Elliptic functions. Springer New York, 1987.

\bibitem{KubertLang}
Kubert, Daniel S.; Serge Lang: Modular units. Springer New York, 1981.

\bibitem{Ma}
Manin, Ju. I.:
Parabolic points and zeta functions of modular curves. (Russian)
Izv. Akad. Nauk SSSR Ser. Mat. 36 (1972), 19-66.

\bibitem{Sch}
Scholl, Anthony J: An introduction to Kato's Euler systems. London Mathematical Society Lecture Note Series (1998), 379-460.

\bibitem{SV}
Sharifi, Romyar; Venkatesh, Akshay: Eisenstein cocycles in motivic cohomology. Preprint 2023 arXiv:2011.07241.

\bibitem{Shioda}
Shioda, Tetsuji: On elliptic modular surfaces, J. Math. Soc. Japan, Vol. 24, No. 1, 1972.

\bibitem{Silverman}
Silverman, J.: Advanced topics in the arithmetic of elliptic curves, Graduate Texts in Mathematics (GTM, volume 151).

\bibitem{Sus}
Suslin, Andrey Aleksandrovich: Reciprocity laws and the stable rank of polynomial rings. Izvestiya Rossiiskoi Akademii Nauk. Seriya Matematicheskaya 43.6 (1979), 1394-1429.

\bibitem{Totaro} Totaro, Burt: Milnor $K$-theory is the simplest part of algebraic $K$-theory. $K$-theory 6.2 (1992), 177-189.

\bibitem{Weibel} Weibel, Charles: The K-book: an introduction to algebraic K-theory, Graduate Studies in Math. vol. 145, AMS, 2013; 
the online version is available at https://sites.math.rutgers.edu/~weibel/Kbook/Kbook.III.pdf.

\bibitem{Xu} Xu, Peter: Arithmetic Eisenstein theta lifts. PhD thesis, McGill University, 2023.

%
%
%
%
%
%
%
%
%
%
%
%
%
%
%
%
%
%
%
%
%
%
%

\end{thebibliography}
\end{document}